\documentclass[reqno]{amsart}
\usepackage[utf8]{inputenc}
\usepackage[margin = 1in]{geometry}
\usepackage{latexsym,amsmath,color,amsthm,epsfig,mathrsfs,mathdots}
\usepackage{mathtools}
\mathtoolsset{showonlyrefs}
\usepackage{graphicx}
\usepackage{amssymb}
\usepackage{mdframed}
\usepackage{fancyhdr}
\usepackage{tikz-cd}
\usepackage{verbatim}
\usepackage{nicefrac}
\usepackage{bbm}
\usepackage{hyperref}
\hypersetup{
    colorlinks = true,
    citecolor=black,
    filecolor=black,
    linkcolor=black,
    urlcolor=blue
}
\usepackage{amsfonts, esint}
\usepackage{color}
\usepackage{dsfont}

\numberwithin{equation}{section}

\usepackage{lastpage}
\newtheorem{thm}{Theorem}[section]
\newtheorem{prop}[thm]{Proposition}
\newtheorem{cor}[thm]{Corollary}

\newtheorem{lem}[thm]{Lemma}

\newtheorem{preremark}[thm]{Remark}

\newenvironment{remark}{\begin{preremark}\rm}{\medskip \end{preremark}}
\numberwithin{equation}{section}

\newcommand{\norm}[1]{\left\Vert#1\right\Vert}

\newcommand{\R}{\mathbb R}
\newcommand{\eps}{\varepsilon}

\newcommand{\grad} {\nabla}
\newcommand{\lap} {\Delta}
\newcommand{\bdary} {\partial}

\newcommand{\dd} {\; \mathrm{d}}

\newcommand{\Jap}[1]{\langle #1 \rangle}

\DeclareMathOperator{\dv}{div}

\definecolor{sh}{RGB}{255,0,100}

\newcounter{case}
\renewcommand{\thecase}{\Alph{case}}
\newcounter{proofcase}[case]
\renewcommand{\theproofcase}{(\thecase\arabic{proofcase})}
\newif\ifusedcase

\newcommand{\proofcase}{%
  \ifusedcase\else\usedcasetrue\stepcounter{case}\fi
  \par
  \refstepcounter{proofcase}
  \everypar=\expandafter{\the\everypar{\setbox0=\lastbox}\everypar{}Case \theproofcase\ }%
}

\begin{document}

\begin{abstract}
    Following the recent ideas of Guillen and Silvestre in \cite{guillen2023global}, we prove that the Fisher information is non-increasing along the flow of the isotropic Landau equation. We then use this fact to deduce global existence for the equation
    $\partial_t f = (-\Delta)^{-1}f \cdot \Delta f + f^2$ under a relatively lax set of conditions on the initial data. In particular, we remove the restrictive radially decreasing assumption of previous works. 
\end{abstract}

\title{Global Existence for an Isotropic Landau Model}
\author{David Bowman and Sehyun Ji}
\address[David Bowman, Sehyun Ji]{Department of Mathematics, University of Chicago,  Chicago, Illinois 60637, USA}
\email{dbowman@uchicago.edu, jise0624@uchicago.edu}
\maketitle

\section{Introduction}
This paper is concerned with the global existence of smooth solutions to the isotropic Landau equation
\begin{equation}
    \label{eq: Krieger-Strain,divform2}
    \partial_t f = \partial_{v_i}\int_{\R^3} |v-w|^{2+\gamma}(\partial_{v_i}-\partial_{w_i})[f(v)f(w)] \dd w
\end{equation}
for some parameter $\gamma \in [-3,-2]$ and a non-negative function $f= f(t,v): [0,\infty) \times \R^3 \to [0,\infty)$. We shall also be interested in the non-divergence form

\begin{equation} \label{eq: Kriegerpotential}
\partial_t f = a[f]\Delta f -(2+\gamma) h[f] f
\end{equation}
where 
$$a[f] \coloneqq f \ast |\cdot|^{2+\gamma}, \ \ \ h[f] \coloneqq \begin{cases} (3+\gamma)  f\ast |\cdot|^{\gamma} & \gamma \in (-3,-2] \\ 4\pi f & \gamma=-3.\end{cases}.$$

After a simple time rescaling, \eqref{eq: Kriegerpotential} in  the most singular case $\gamma=-3$ takes the form
\begin{equation}
    \label{eq: Krieger-Coulomb}
\partial_t f = (-\Delta)^{-1}f \cdot \Delta f +f^2.
\end{equation}
This was the equation initially proposed in \cite{krieger2012ks} by Krieger and Strain, after whom we term \eqref{eq: Krieger-Strain,divform2} the \textit{Krieger-Strain equation}, as a modification of the \textit{Landau-Coulomb equation}
\begin{equation}
    \label{eq: Landau-Coulomb}
    \partial_tf = \partial_{ij}(-\Delta)^{-2}f \cdot  \partial_{ij}f + f^2.
\end{equation}
For this reason we refer to the case $\gamma=-3$ of the Krieger-Strain equation as the \textit{Coulomb} case. The Landau-Coulomb equation \eqref{eq: Landau-Coulomb} is one of the classical equations of physics that has drawn a great deal of attention from the mathematical community over the last half of a century. A satisfying understanding of \eqref{eq: Landau-Coulomb} seemed far out of reach until recently, when Guillen and Silvestre made an important breakthrough and resolved the question of global existence using an innovative and novel viewpoint for the equation, see \cite{guillen2023global}. 

Our approach for analyzing \eqref{eq: Krieger-Strain,divform2}  follows the new ideas from \cite{guillen2023global}. Taking their method as inspiration, our main tool for proving global existence of solutions to \eqref{eq: Krieger-Strain,divform2} is the Fisher information. For a positive function $f:\R^3 \to (0,\infty)$, the Fisher information of $f$ is given by 
\begin{equation}
\label{eq: Fisherinfo}
i(f) \coloneqq \int_{\R^3} \frac{|\nabla f|^2}{f} \dd v= \int_{\R^3}|\nabla \log f|^2 f  \dd v= 4\int_{\R^3} |\nabla \sqrt{f}|^2 \dd v.
\end{equation}
It can be extended for $f$ with vacuum regions by defining $|\grad f|^2/f=0$ on those points. From the third form of the Fisher information in \eqref{eq: Fisherinfo} and the Sobolev embedding $H^1(\R^3)\subset L^6(\R^3)$, it is clear that control of $i(f)$ implies the control of $\|f\|_{L^3}$. 

One of the purposes of this paper is to extend the class of evolution equations whose solutions have monotone decreasing Fisher information, which is an extremely powerful tool for obtaining regularity estimates. At least in the folklore of analysis, it has been known for more than half a century that the Fisher information decreases along the flow of the heat equation, for instance see \cite{villani2000fisherheat}.
Next, in the 1960s, McKean obtained results in the Maxwell Molecules case (corresponding to $\gamma=0$ in \eqref{eq: Krieger-Strain,divform2}) for a one dimensional toy model of the Boltzmann equation known as the Kac equation, see  \cite{mckean1966kac}. 
The next progress wasn't until the 1990, when Toscani proved the monotonicity of the Fisher information for the $2D$ Boltzmann equation with Maxwell Molecules \cite{toscani1992boltzmann}.
After a few years later, Villani proved the monotonicity in the Maxwell Molecules case  of the Landau and Boltzmann equations \cite{villani1998fisherboltzmann, villani2000fisherlandau}.  Finally, very recently, Guillen and Silvestre extended the result to the Landau-Coulomb equation \eqref{eq: Landau-Coulomb} and the broader range of the Landau equation \eqref{eq : Landau, div}. See \cite{guillen2023global} and the references therein for further discussion of the history of the Fisher information. Overall, very few equations are known to decrease the Fisher information. We prove the following theorem, which expands this class of equations to include \eqref{eq: Krieger-Strain,divform2}. 
\begin{thm}
    \label{thm: FisherTheorem} Let $f:[0,T] \times \R^3 \to [0,\infty)$ be a classical solution with rapid decay to the equation
    \begin{equation}
    \label{eq: generalizedKrieger}
        \partial_tf = \partial_{v_i} \int_{\R^3} \alpha(|v-w|) |v-w|^2 (\partial_{v_i}-\partial_{w_i})[f(v)f(w)] \dd w
    \end{equation}
    where $\alpha:[0,\infty) \to [0,\infty)$ is an arbitrary interaction potential. Assume that $\alpha(\cdot)$ satisfies 
    $$\frac{r\alpha'(r)}{\alpha(r)}\in [2-3\sqrt{3}, -2+2\sqrt{2}] \ \ \text{ for all } r>0.$$
    Then the Fisher information $i(f)$ is monotone decreasing as a function of time. In particular, taking $\alpha(r) = r^\gamma$, we recover the monotonicity of the Fisher information for solutions to \eqref{eq: Krieger-Strain,divform2} for $\gamma \in [2-3\sqrt{3}, -2+2\sqrt{2}].$
\end{thm}
Note that $2-3\sqrt{3} \simeq -3.196$, and $-2+2\sqrt{2} \simeq 0.828.$ It is only when studying the Fisher information that we consider the general interaction potential $\alpha(\cdot)$, which has a more compact notation. It is interesting to note that in the case $\gamma=-2$, the Krieger-Strain equation \eqref{eq: Krieger-Strain,divform2} is simply the heat equation, thanks to the conservation of $L^1$ norm. Therefore one can view the above theorem as starting at the known result for the heat equation and extending to the more complicated Coulomb case \eqref{eq: Krieger-Coulomb}.

We then seek to exploit the newfound control of $\|f(t)\|_{L^3}$ to obtain a global existence result for \eqref{eq: Krieger-Strain,divform2}. Let us focus our discussion on the Coulomb case. Whether \eqref{eq: Krieger-Coulomb} should have smooth solutions has not always been so clear. The equation bears resemblance to the semilinear heat equation with quadratic nonlinearity
$$\partial_t u = \Delta u + u^2,$$
which is known to exhibit finite-time blowup. However, the diffusion operator in \eqref{eq: Krieger-Coulomb} seems to have a stronger regularization effect than the heat operator, which might prevent the emergence of a blow-up of the solution to \eqref{eq: Krieger-Coulomb}. Indeed, around points where the mass of $f$ may be accumulating, the diffusion coefficient  $(-\Delta)^{-1}f$ grows larger, competing with the reaction term $f^2$ to drive the solution downward.  
The primary contribution of our paper is to prove that in fact smooth solutions do exist globally. 

\begin{thm}
\label{thm: globalexistencethm}
    Let $f_{\text{in}}:\R^3 \to [0,\infty)$ be an initial data that belongs to $L^\infty(\R^3) \cap L^1_m(\R^3)$ for sufficiently large $m>0$. Assume also that the initial data has finite Fisher information, i.e.,
    $$i(f_{\text{in}}) = \int_{\R^3} |\nabla \log f_{\text{in}}|^2 f_{\text{in}} \dd v<\infty.$$
    Then there exists a classical global solution $f:\R^3 \to [0,\infty)$ to the Krieger-Strain equation \eqref{eq: Krieger-Strain,divform2} which is $C^\infty$ and strictly positive for positive time $t>0$. 
    \end{thm}

As we shall see, the uniform estimates on $\|f(t)\|_{L^3}$ provided by Theorem \ref{thm: FisherTheorem} will allow us to apply the methods and results of Silvestre in \cite{silvestre2017upper} to obtain stronger ellipticity bounds on $(-\Delta)^{-1}f$, and thus $L^\infty$ estimates on $f$. This is the crucial step towards preventing blow-up and proving the global existence of smooth solutions. 

Up until now, the only global existence results for \eqref{eq: Krieger-Strain,divform2} have been either for radially symmetric and decreasing data. In fact, there is no pre-existing local existence theory or continuation criterion for the Krieger-Strain equation which does not make the strong structural assumptions of radial symmetry and monotonicity. In order to have a satisfying global existence theory, we also prove the following local existence theorem: 

\begin{thm}
\label{thm: Linftylocaltheory}
    Suppose that the initial data $f_{\text{in}}$ is non-negative and that $f_{\text{in}} \in L^\infty(\R^3) \cap L^1_m(\R^3)$ for some sufficiently large $m>0$. 
    Then there exists a classical solution $f:[0,T_*) \times \R^3 \to [0,\infty)$ of the Krieger-Strain equation \eqref{eq: Kriegerpotential}, attaining the initial data in the weak-$\ast$ $L^\infty$ sense. Additionally, if we assume further that the Fisher information of the initial data is finite, i.e., that $i(f_{\text{in}})<\infty$, then $i(f(t))$ exists for all $t\in [0,T_*)$ and is monotone decreasing as a function of time. Lastly, the maximal time of existence $T_*$ is such that if $T_*<+\infty$, then 
    $$\lim_{t\nearrow T_*} \|f(t)\|_{L^\infty} =+\infty.$$
\end{thm}
The method is standard, following the approach of Golding and Loher \cite{golding2024localintime}. First we establish a growth estimate on $\|f(t)\|_{L^\infty}$ that depends only on $\|f_{\text{in}}\|_{L^\infty}$. We then construct solutions for smooth and decaying data $f_{\text{in}}^\eps$ (see Appendix $B$) which approximate the initial data $f_{\text{in}}.$  Via standard regularity estimates, we are able to prove the solutions corresponding to the $\{f_{\text{in}}^\eps\}$ converge in a strong sense to a smooth solution $f$ with the right initial data. 
\subsection{History and Review of the Krieger-Strain equation}
As previously mentioned, the Coulomb case of Krieger-Strain \eqref{eq: Krieger-Coulomb} was introduced as a modification of the celebrated Landau-Coulomb equation \eqref{eq: Landau-Coulomb}. In the same way, one may consider \eqref{eq: Krieger-Strain,divform2}, or more generally \eqref{eq: generalizedKrieger}, as a modification of the Landau equation 

\begin{equation} \label{eq : Landau, div}
\partial_t f  = \partial_{v_i} \int_{\R^3} \alpha(|v-w|) a_{ij}(v-w) (\partial_{v_j}-\partial_{w_j})[f(v)f(w)] \dd w ,
\end{equation}
where $a_{ij}(z) = |z|^2\delta_{ij} - z_iz_j$ is a weighted orthogonal projection onto $z^\perp$, and  $\alpha(\cdot)$ is a radial interaction potential, typically taken to be $\alpha(r) = r^\gamma$. There are several key differences that complicate the transference of results from the Landau equation to Krieger-Strain equation. The Landau equation features a conservation of the energy, defined below, while this is not the case with \eqref{eq: Krieger-Strain,divform2}. A lack of any a priori higher moment estimates for the Krieger-Strain equation complicates the matter of obtaining unconditional $L^\infty$ estimates following the method of Silvestre in \cite{silvestre2017upper}. To see this, let us define respectively the mass, energy, and entropy of a solution $f$ at time $t$ by

\begin{align}
    M(t) &\coloneqq \int_{\R^3} f(t,v) \dd v, \\ 
    E(t) &\coloneqq  \int_{\R^3} f(t,v)|v|^2 \dd v, \\ 
    H(t) &\coloneqq  \int_{\R^3} f(t,v) \log f(t,v) \dd v.
\end{align}

One may see in exactly the same manner that solutions to both the Landau and Krieger-Strain equations preserve mass ($\frac{d}{dt}M(t)=0$) and dissipate entropy ($\frac{d}{dt}H(t)\le 0$), but energy is only conserved in the former case. In fact, the flow of \eqref{eq: Krieger-Strain,divform2} increases energy, which can be seen with the following simple computation:
\begin{align} \frac{d}{dt} E(t) &= \int_{\R^3} |v|^2 \partial_t f \dd v= \int_{\R^d} |v|^2 \partial_{v_i} \int_{\R^3} |v-w|^{2+\gamma} \partial_{v_i-w_i}(f(v)f(w))\dd w \dd v\\
&= \int_{\R^{6}} 2v_i|v-w|^{2+\gamma} \partial_{v_i-w_i}(f(v)f(w))\dd w \dd v\\
&= \int_{\R^{6}} (v_i-w_i) |v-w|^{2+\gamma} \partial_{v_i-w_i}(f(v)f(w))\dd w \dd v\\
&= \int_{\R^{6}} (5+\gamma) |v-w|^{2+\gamma}f(v)f(w) \dd w \dd v.\\
\end{align}
With the notation of \eqref{eq: Kriegerpotential}, we find that the energy increases with rate
\begin{equation}
\label{eq: energygrowth} \frac{d}{dt} E(t) = (5+\gamma) \int_{\R^3} a[f(t,v)]f(t,v) \dd v,
\end{equation}
a fact that was first observed in \cite{gualdani2022hardy}. That there is an increase in energy is a departure from the Landau setting and poses an additional roadblock toward obtaining estimates on the solution. Heuristically, entropy dissipation implies that the $L^1$ mass of the solution cannot concentrate on too small of a set, while the energy bound implies that the $L^1$ mass may not escape to infinity. Both are necessary in order to prove ellipticity estimates on the diffusion coefficient $a[f]$. In Section $5$, the question of controlling the energy will be explored further, and in fact there shall be obtained higher moment estimates of all orders. We emphasize that these estimates will only made possible due to the monotonicity of the Fisher information from Theorem \ref{thm: FisherTheorem} and the $L^3$ bound that it implies. 

One of the major areas of investigation for the Landau equation \eqref{eq : Landau, div} is the convergence to an equilibrium as $t \to +\infty$. The Landau equation features a rich family of stationary solutions of the form $\mathcal{M}(v) = a\exp(-b|v-u|^2)$ for some $a,b \ge 0, u \in \R^3$. Such steady states are called \textit{Maxwellians}, and there is a rich history of papers exploring the convergence to Maxwellians of solutions to \eqref{eq : Landau, div}, see for instance \cite{carrapatoso2017landau, desvillettes2005equilibrium, guo2002periodic, strain2006almostexponential, strain2008exponential}. The Krieger-Strain equation \eqref{eq: Kriegerpotential} has only one stationary solution: $f \equiv 0$. Let us briefly explore whether a solution $f(t,v)$ to \eqref{eq: Kriegerpotential} may converge to the steady state of zero as $t \to +\infty$, and see how this question is connected to the increase of energy \eqref{eq: energygrowth}. If the energy $E(t)$ stays bounded as $t\to \infty$ (which is the case for solutions to the Landau equation), then it is impossible for the convergence $$\displaystyle \lim_{t\to \infty}f(t,v)=0$$ to take place point-wise almost everywhere. Indeed, the energy being bounded would provide uniform integrability of the family $\{f(t)\}_{t\ge 0} \subset L^1(\R^3)$, which combined with pointwise convergence implies strong $L^1$ convergence. This would contradict the conservation of mass $M(t) \equiv M(0) \neq 0$. In this way we may see that the long-time behavior for the Krieger-Strain equation can differ drastically from that of the Landau equation.

Lastly, we point out that when $\gamma \in (-2,1]$, the Krieger-Strain equation departs from the Landau equation. When written in non-divergence form (see \cite{silvestre2017upper}, the Landau equation features a reaction term of the form $c_{d,\gamma}[f\ast|\cdot|^\gamma]f$, where $c_{d,\gamma}>0$. The non-divergence form of the Krieger-Strain equation, on the other hand, features the reaction term $-(2+\gamma)(3+\gamma) [f\ast|\cdot|^\gamma] f$. For $\gamma >-2$, the reaction term becomes negative, and thus no longer drives the solution upwards. This is a dramatic qualitative difference that greatly simplifies the analysis. Still, we remark that all of the theorems developed in this paper may be extended to a larger range of $\gamma$, with very similar proofs. We elect to omit the analysis of this range of $\gamma$ in order to have the cleanest presentation possible.

Now we review the history of the Krieger-Strain equation. In 2011, Krieger and Strain \cite{krieger2012ks} introduced the equation 
$$\partial_t f = (-\Delta)^{-1}f \cdot \Delta f + \delta f^2, \ \ \ \delta \in [0,1]$$
in order to better understand the Landau-Coulomb equation, as previously mentioned. The point of the $\delta$ was to discount the strength of the quadratic nonlinearity, which is known to cause blow-up for the semilinear heat equation. Here they proved global existence for radially symmetric, bounded, and decreasing data, for $\delta \in [0,\frac23]$. Shortly after, Gressman, Krieger, and Strain improved the range to $\delta \in [0,\frac{74}{75}]$ in \cite{gressman2012ks}. Half a decade later, Gualdani and Guillen reach the goal of $\delta=1$, still in the radially decreasing setting. Finally, in 2021, Gualdani and Guillen were able to prove classical solutions to \eqref{eq: Krieger-Strain,divform2} remain smooth for as long as they exist, for $\gamma \in [-\frac{5}{2}, -2]$, using a novel approach based on Hardy's inequality and the propagation of certain $L^p$ norms. While this contribution was notable in that it was the first time the radial symmetry condition was dropped, the Hardy's inequality approach fails as $\gamma \to -3^+$. Finally, we call attention to the work of Gualdani and Zamponi in \cite{gualdani2018evensol}, who managed to prove a global existence result for initial data which are merely bounded and even. With the exception of \cite{gualdani2022hardy}, the existence results in the literature make weak assumptions on the regularity of the data, but strong assumptions on the structure. Noticeably absent from the literature has been a local existence theory for more generic initial data, or a continuation criterion for extending solutions. 

We also point out the interesting work of Snelson on the isotropic Boltzmann equation \cite{snelson2023global}. In this work a global existence result was obtained for a modification of the Boltzmann equation which is analogous to the Krieger-Strain equation. Because the range of $\gamma$ considered for the Boltzmann equation does not reach the particularly singular range $\gamma \simeq -3$, it was possible for Snelson to argue based on the Hardy's inequality approach of Gualdani and Guillen in \cite{gualdani2022hardy}. His approach has the same limitations, and so new ideas are needed. 
\subsection{Outline of Paper and Main Ideas}
The first half of the paper is devoted to showing the Fisher information decreases along the flow of the Krieger-Strain equation \eqref{eq: Krieger-Strain,divform2}. Our method is inspired by the recent novel approach of Guillen and Silvestre \cite{guillen2023global}. Let us briefly explain their approach, as it is the starting point for ours. 
The starting point is the observation that the Fisher information interacts nicely with tensor products $f\otimes f(v,w) \coloneqq f(v)f(w)$. Indeed, defining $\displaystyle I(F) \coloneqq \iint_{\R^6} \frac{|\nabla F|^2}{F} \dd w \dd v$ to be the \textit{lifted} Fisher information for $F(v,w) :\R^6 \to \R$, it holds for any probability density function $f$ that
$$i(f) = \frac12 I(f\otimes f).$$

The authors of \cite{guillen2023global} then observe that it suffices to consider the Fisher information along the flow of the equation in doubled variables
$$\begin{cases} \partial_tF = Q_L(F), \\ F(0,v,w) = f\otimes f(v,w), \end{cases}$$
where $f$ solves \eqref{eq: generalizedKrieger} and the lifted Landau collision operator $Q_{L}$ is a linear, degenerate elliptic operator with \textit{local} coefficients that can be explicitly computed. The simplified form of $Q_L$ makes it considerably more advantageous to study. More details will be given in Section $2$. Continuing, they then decompose $Q_L(F) = \sum_{k=1}^3 L_k \circ L_k(F)$, where the $L_k$ are transport operators in directions parallel to the level sets of $|v-w|$. In other words, the flows of the operators $L_k$ are rotations when restricted to level sets of $|v-w|$, but not necessarily when moving across the layers in the normal direction. A more precise description will be given in Section $2$. This new formulation allows for a sort of Lagrangian perspective on the Landau equation and the Fisher information, based on tracking the flows of particles transported by the $\{L_k\}$. With this viewpoint in mind they decompose the Fisher information with $I= I_{\text{tan}} + I_n$, where $I_{\text{tan}}$ takes into account derivatives of $F$ in directions tangent to the level sets of $|v-w|$, and $I_n$ derivatives of $F$ in the direction $n$ normal to the level sets. At this point, they make the basic observation that the Fisher information is preserved by rotation. With this fact in mind, Guillen and Silvestre argue that from the perspective of $I_{\text{tan}}$, the flow of each $L_k$ is rotational and therefore $I_{\text{tan}}$ is constant along the flow of each $L_k$. By a simple convexity argument it is proven that $I_{\text{tan}}$ decreases along the flow of the lifted Landau equation. A concrete estimate is obtained for its dissipation. 

The majority of the analysis in \cite{guillen2023global} is dedicated to analyzing the Fisher information in the normal direction, $I_n$. There is no reason a priori to expect any monotonicity of $I_n$, as the flows of $L_k$ are no longer isometric as one moves across layers in the normal direction. Performing an even more delicate decomposition, Guillen and Silvestre reduce the problem to proving a Poincare-like inequality on the sphere $\mathbb{S}^2$, ultimately proving that the Fisher information is monotone decreasing for all radial interaction potentials $\alpha(\cdot)$ with $\frac{|r\alpha'(r)|}{\alpha(r)} \le \sqrt{19}$. 

Our approach for proving Theorem \ref{thm: FisherTheorem} is to make use of the analysis done in \cite{guillen2023global}. Performing the same lifting procedure for \eqref{eq: Krieger-Strain,divform2}, we obtain that the ``lifted" solutions satisfy an equation of the form
$$\partial_t F = Q_L(F) + R(F)$$
for some remainder operator $R$. The remainder operator $R$ involves first and second order derivatives in the direction normal to the level sets of $|v-w|$. In the Landau setting, one must only account for one derivative in the normal direction, which is that coming from the definition of the Fisher information. The Landau operator itself does not involve derivatives normal to the level sets of $|v-w|$. In our case, the remainder operator $R$ is particularly undesirable in that it involves first and second order derivatives in the normal direction. The entirety of our proof is spent showing the contribution of this term is not too disadvantageous. Through careful computation, we are able to show that the estimate from \cite{guillen2023global} is just barely strong enough to control the error term $R$ in the Coulomb case of Krieger-Strain \eqref{eq: Krieger-Coulomb}. In Section $2$, we set forth the groundwork for the lifting procedure, and in Section $3$ we perform the computation in the simplest \textit{Maxwell-Molecules} case (corresponding to $\gamma=0$ in \eqref{eq: Krieger-Strain,divform2} or $\alpha \equiv 1$ in \eqref{eq: generalizedKrieger}), which captures the spirit of the computation without the technicality of the general case. Finally, in Section $4$, we prove the Fisher information decreases for a wide range of interaction potentials, including \eqref{eq: Krieger-Coulomb}. 

Since the literature for the equation is relatively sparse, we must build from the ground up. The monotonicity of the Fisher information forms the foundation for everything that follows in the proof of global existence of solutions to \eqref{eq: Krieger-Strain,divform2}. Sections $2-4$ are dedicated to the analysis of the Fisher information. It is only in these sections that we use the framework of the lifting procedure outlined above. In Section $5$, we prove new growth estimates on the moments $\int_{\R^3} f(t,v) |v|^s \dd v$, which enables us to apply the methods of Silvestre \cite{silvestre2017upper} to obtain an a priori estimate on $\|f(t)\|_{L^\infty}$. In Section $6$ we obtain delicate estimates on the propagation of Sobolev norms $\|f(t)\|_{H^k}$ for every $k \in \mathbb{N}$, which will be crucial towards proving the solution does not form implosion singularities. Finally in Section $7$, we establish an $L^\infty \cap L^1_m$ local existence theory, without making any further assumptions on the structure of the solution. In order to do this, we first construct solutions to \eqref{eq: Kriegerpotential} for initial data lying in $H^3_q(\R^3)$. This is merely an intermediate step whose proof is standard, so it is relegated to Appendix B. With all of these pieces put together, the global existence result follows straightforwardly in Section $8$. Finally, some key convolution and interpolation estimates may be found in Appendix A.

\subsection{Acknowledgements}
The authors would like to thank Luis Silvestre for suggesting this problem and for many illuminating discussions.

\section{Preliminary Set-up}
We follow the approach of \cite{guillen2023global} by lifting \eqref{eq: generalizedKrieger} onto $\R^6$ and decomposing it into the sum of squares of flows along certain vector fields, plus a first-order correction term. Let us quickly summarize the framework of \cite{guillen2023global} and recall the main lemmas. Given $f: \R^3 \to \R$, one may define the tensor product $F: \R^{3+3} \to \R$ by 
$$F(v,w) \coloneqq (f\otimes f)(v,w) = f(v)f(w)$$
which turns a scalar function on $\R^3$ into a scalar function on $\R^{3+3}$. One may go the other direction by projecting onto the marginal with 
$$\pi F(v) \coloneqq \int_{\R^3} F(v,w) \dd w.$$

Let us recall from \eqref{eq: generalizedKrieger} that the Krieger-Strain equation with a general radial potential $\alpha(\cdot)$ is written as 
\begin{equation}
\label{eq: generalizedKREIGER2}
    \partial_tf = \partial_{v_i} \int_{\R^3} \alpha(|v-w|) |v-w|^2 (\partial_{v_i}-\partial_{w_i})[f(v)f(w)] \dd w \coloneqq \mathcal{Q_{KS}
    }(f,f).
\end{equation}
Ultimately we will concern ourselves with $\alpha(r) = r^\gamma$. The advantage of this formulation is that it allows to more easily compare with the form of the Landau equation \eqref{eq : Landau, div} considered in \cite{guillen2023global}. 

Now we perform the lifting procedure of Guillen and Silvestre for \eqref{eq: generalizedKREIGER2}. First, we define the lifted Krieger-Strain collision operator 
\begin{align} \label{e: KrigerLift1} Q_{KS}(F)(v,w) \coloneqq (\partial_{v_i} - \partial_{w_i})[(\alpha(|v-w|)|v-w|^2(\partial_{v_i}-\partial_{w_i})] F(v,w)].
\end{align}

For a function $f(v)$, we consider the initial value problem
\begin{equation}
  \label{eq: LIFTEDKRIEGER}
    \begin{cases}       
    \partial_t F = Q_{KS}(F),\\
    F(0,v,w)=f(v)f(w).
    \end{cases}  
\end{equation}

As in \cite{guillen2023global}, we have doubled the variables to obtain a linear equation with local and explicit coefficients.  

The following sequence of lemmas is literally the exact same as in \cite{guillen2023global}. The point of the lemmas is to show that it suffices to consider the ``lifted" Fisher information  
$$I(F)\coloneqq \iint_{\R^6} \frac{|\nabla F|^2}{F} \dd w \dd v = \iint_{\R^6}|\nabla \log F|^2 F \dd w \dd v = 4\iint_{\R^6} |\nabla \sqrt{F}|^2 \dd w  \dd v$$
along the flow of \eqref{eq: LIFTEDKRIEGER} in order to prove the Fisher information $i(f(\cdot))$ decreases along the flow of \eqref{eq: Krieger-Strain,divform2}. Only Lemmas \ref{lem1} and \ref{lem4} have anything to do with an equation, and the proofs are the exact same as in the Landau case. 
\begin{lem} \label{lem1}
    For any twice-differentiable function $f$ which is nonnegative and has $\int_{\R^3} f=1$, the Krieger-Strain operator $\mathcal{Q}_{KS}$, defined as in \eqref{eq: generalizedKREIGER2}, coincides with 
    $$ \mathcal{Q}_{KS}(f,f) = \pi(Q_{KS}(F)) = \partial_t[\pi F]_{\bigm|_{t=0}}$$
    where $F$ is the solution to \eqref{eq: LIFTEDKRIEGER}.
\end{lem} \label{lem2}
The proof is the exact same as in \cite{guillen2023global}. Note that we can rescale the equation \eqref{eq: generalizedKREIGER2} to make $\int_{\R^3} f=1$.
\begin{lem}
    For any non-negative $C^1$ function $f$ with $\int_{\R^3} f =1$ and $i(f)<\infty$, we have
    $$i(f) = \frac12 I(f\otimes f).$$
\end{lem}

\begin{lem} \label{lem3}
    For any non-negative $C^1$ function $F: \R^6 \to [0,\infty)$ with $F(v,w) = F(w,v)$, we have
    $$i(\pi F) \le \frac12 I(F).$$
\end{lem}
\begin{lem} \label{lem4}
    Let $f$ be a solution of \eqref{eq: generalizedKREIGER2} and let $F$ be a solution of \eqref{eq: LIFTEDKRIEGER}, then 
    $$\partial_t i(f) \le \frac12 \partial_t I[F]_{\bigm|_{t=0}}$$
\end{lem}
The preceding sequence of lemmas from \cite{guillen2023global} shows that it suffices to prove that the Fisher information $I[\cdot]$ on $\R^6$ decreases along the flow of the lifted equation \eqref{eq: LIFTEDKRIEGER}. Our next order of business is to compare $Q_{KS}$ with the operator $Q_{L}$ obtained in \cite{guillen2023global}, since we shall use the analysis of $Q_L$ extensively. First they obtain the expression
$$Q_L(F) \coloneqq \sum_{k=1}^3 (\partial_{v_i}-\partial_{w_i}) \alpha(|v-w|) (b_k)_i(b_k)_j \cdot(\partial_{v_j}-\partial_{w_j})F,$$
where the $b_k = b_k(v-w)\in \R^3$ are given by
$$b_1(v-w) = \begin{pmatrix} 0 \\ w_3-v_3 \\ v_2-w_2\end{pmatrix}, \ \ b_2(v-w) = \begin{pmatrix} v_3-w_3 \\ 0 \\ w_1-v_1\end{pmatrix}, \ \ \text{and } b_3(v-w) = \begin{pmatrix} w_2-v_2 \\ v_1-w_1 \\ 0\end{pmatrix}.$$

The vectors $\{b_k(v-w)\}$ span $(v-w)^\perp$ and satisfy the identity
$$\sum_{k=1}^3 b_k \otimes b_k = a_{ij}(v-w),$$
where $a_{ij}(v-w) = |v-w|^2\delta_{ij} - (v_i-w_i)(v_j-w_j)$ as in \eqref{eq : Landau, div}. In particular, the $\{b_k\}$ span the tangent space to the level sets of $|v-w|$. They then proceed to set
$$\tilde{b}_k(v,w) \coloneqq \begin{pmatrix} b_k(v-w) \\ -b_k(v-w)\end{pmatrix}$$
to arrive at
\begin{equation}
\label{eq: expressionforQ_L}
    Q_L(F) \coloneqq \sum_{k=1}^3 \sqrt{\alpha}\tilde{b}_k \cdot \nabla(\sqrt{\alpha}\tilde{b}_k  \cdot \nabla F).
\end{equation}

Here $\nabla$ refers to the gradient in the full variables $(v,w) \in \R^6$. In words, they express the lifted Landau collision operator $Q_L$ as the sum of squares of first order transport operators $\sqrt{\alpha}\tilde{b}_k\cdot \nabla$, where $\tilde{b}_k$ always lies tangent to the level sets of $\alpha(\cdot)$. Here and throughout the paper, we abuse notation with $\alpha$, thinking of it also as a function of $(v,w)\in \R^6$. Now we return to the Krieger-Strain setting by defining 
$$b_0(v-w) \coloneqq v-w \in \R^3,$$ for which it holds that 
\begin{equation}
    |v-w|^2 \text{Id}_3 = a_{ij}(v-w) + b_0(v-w)\otimes b_0(v-w),
\end{equation} where $a_{ij}$ is as in \eqref{eq : Landau, div}. We remark that $b_0(v-w)$ points in the direction normal to the level sets of $|v-w|$, which is a departure from the Landau setting. It is immediate that \eqref{e: KrigerLift1} becomes

\begin{align} \label{e: KriegerLift2} Q_{KS}(F) = Q_L(F) + (\partial_{v_i}-\partial_{w_i})[\alpha(|v-w|) (b_0)_i(b_0)_j (\partial_{v_j}-\partial_{w_j})F(v,w)].
\end{align}

Now we complete the lifting by setting 
\begin{equation}
    \tilde{b}_0(v,w) \coloneqq \begin{pmatrix} v-w \\ w-v \end{pmatrix} \in \R^6 \label{eq: b0def}
\end{equation}
at which point we see \eqref{e: KriegerLift2} may be written as
\begin{align} \label{e: KriegerLift3} Q_{KS}(F) &= Q_L(F) + \nabla \cdot (\alpha(v-w) \tilde{b}_0 \otimes \tilde{b}_0 \nabla F)\\
 &=Q_L(F) +  \nabla \cdot (\sqrt{\alpha} \tilde{b}_0  (\sqrt{\alpha}\tilde{b}_0 \cdot \nabla F)).
 \end{align}
where now the divergence and gradient are in the full variables $(v,w)$. We also define the normalized $\tilde{b}_0$ as $n$, which is
\begin{equation}
    n\coloneqq \frac{1}{\sqrt{2}|v-w|} \begin{pmatrix} v-w \\ w-v \end{pmatrix} \in \R^6.
\end{equation}
We record for posterity that 
$$\nabla \sqrt{\alpha} = \frac{\grad \alpha}{2\sqrt{\alpha}}= \frac{\alpha'}{2\sqrt{\alpha}}\begin{pmatrix} \frac{v-w}{|v-w|} \\ \frac{w-v}{|v-w|}\end{pmatrix} = \frac{\alpha'}{2\sqrt{\alpha}} \frac{\tilde{b}_0}{|v-w|}=\frac{\alpha'}{\sqrt{2\alpha}} n.$$

This further complicates the Krieger-Strain situation. In the Landau setting \eqref{eq: expressionforQ_L}, one may move any powers of $\alpha$ through the differential operators $\tilde{b}_k \cdot \nabla$ for $1\le k \le 3$, since $\nabla \alpha \perp \tilde{b}_k$. Because of this, as we manipulate \eqref{e: KriegerLift3} and pass the derivative inside, we pick up an error term:

\begin{align}  Q_{KS}(F) &= Q_L(F) + \sqrt{\alpha}\tilde{b}_0 \cdot \nabla (\sqrt{\alpha}\tilde{b}_0 \cdot \nabla F) + \nabla \cdot(\sqrt{\alpha}\tilde{b}_0) (\sqrt{\alpha} \tilde b_0 \cdot \nabla F)\\
& = Q_L(F) + \sqrt{\alpha}\tilde{b}_0 \cdot \nabla (\sqrt{\alpha}\tilde{b}_0 \cdot \nabla F) + (\sqrt{\alpha}\dv(\tilde{b}_0) + \frac{\alpha'}{\sqrt{2\alpha} }n\cdot\tilde{b}_0)(\sqrt{\alpha}\tilde{b}_0 \cdot \nabla F)\\
& = Q_L(F) +\sqrt{\alpha}\tilde{b}_0 \cdot \nabla (\sqrt{\alpha}\tilde{b}_0 \cdot \nabla F)  +\left(6\sqrt{\alpha} + \frac{\alpha'}{\sqrt{\alpha}}|v-w|\right) (\sqrt{\alpha}\tilde{b}_0 \cdot \nabla F).\label{eq: KriegerLiftFinal}\end{align}

Here we used that $\dv(\tilde{b}_0)= 6$ in the three dimensional case. This is the form of $Q_{KS}$ that we shall make use of. We will need to consider the derivative of the Fisher information along the flow of each of the three operators in \eqref{eq: KriegerLiftFinal}. In doing so, we will make direct use of the computation for $\langle I'(F), Q_L(F)\rangle$ from \cite{guillen2023global}. 

Before proceeding, it is worthwhile to understand a little flows of the the transport operators $\sqrt{\alpha} \tilde{b}_k\cdot \nabla$ for $0\le k \le 3$. To fix ideas, let us set $\alpha \equiv 1$ (the \textit{Maxwell-Molecules} case). The flows are given by 

\begin{equation}
\begin{pmatrix} \dot{v}(t) \\ \dot{w}(t)\end{pmatrix} = \tilde{b}_j(v,w), \ \ \ (v(0),w(0)) = (v_0,w_0)\in \R^6.
\label{eq: Flowequation}
\end{equation}
As discussed in \cite{guillen2023global},  for the Landau equation $1\le j \le 3$, the following properties hold: 
\begin{itemize}
    \item $v(t) + w(t)$ is constant (preservation of midpoint),
    \item $|v(t)|^2 + |w(t)|^2$ is constant (preservation of norm),
    \item $|v(t)-w(t)|$ is constant (preservation of distance).
\end{itemize}

The facts imply the flows of the vector fields $\tilde{b}_k$ for  $1\le j \le 3$ lie on the two-dimensional Boltzmann sphere
$$\text{sphere}(v,w) \coloneqq \{(v',w') \in \R^6: v'+w' = v+w, |v'|^2 + |w'|^2 = |v|^2 +|w|^2\}.$$

The geometry is different in the Krieger-Strain equation, for we must consider the flow along the vector field $\tilde{b}_0$, which points normal to the level sets of $|v-w|$ in $\R^6$. Solving \eqref{eq: Flowequation} for $\tilde{b}_0$, it is easy to see that midpoints are preserved, but distances and norms are not. In particular, considering points $(v,w) \in \R^6$ as pairs of points $v,w \in \R^3$, the flow of $\tilde{b}_0$ pulls points apart with exponential speed 
$$|v(t)-w(t)|^2 = |v_0-w_0|^2 e^{4t}.$$

This fact complicates the geometry of $Q_{KS}$ as compared to $Q_L$, as our operator is no longer generated by rotations. The vector field $\tilde{b}_0$ is less well-suited to computation than the Landau vector fields $\tilde{b}_k, 1\le k \le 3$. For instance, whereas the derivatives $D\tilde{b}_k$ are anti-symmetric for $1\le k \le 3$, and in particular the $\tilde{b}_k$ are divergence-free, the derivative $D\tilde{b}_0$ takes the form
\begin{equation}
    \label{eq: Db_0matrix}
    D\tilde{b}_0 = \begin{pmatrix} \text{Id}_3& -\text{Id}_3\\ -\text{Id}_3 & \text{Id}_3\end{pmatrix} \in \R^{6\times 6.}
\end{equation}
The lack of the divergence-free structure for $\tilde{b}_0$ leads to more complicated expressions in the analysis of the Fisher information. 

We conclude the preliminary section with some straightforward lemmas about various derivatives of the Fisher information which we shall refer back to throughout the next two sections. One piece of notation we make use of is the commutator $[a,b]$. For given smooth vector fields $a, b: \R^6 \to \R^6$, we define
$$[a,b]\cdot \nabla G \coloneqq a\cdot \nabla (b\cdot \nabla G) - b\cdot \nabla (a\cdot \nabla G).$$

That is, $[a,b]$ measures the error one incurs when switching the order of differentiation along the vector fields $a$ and $b$. There is a more in depth discussion in \cite{guillen2023global}, but we would like to highlight the following fact
\begin{equation}
\label{eq: commutatorfact}
    [a,b] =0 \text{ if and only if } (Da)b = (Db)a.
\end{equation}
This is obtained easily by differentiating. We shall use \eqref{eq: commutatorfact} in Section $4$. Continuing on, we will often need to consider a partial, weighted type of Fisher information
$$I_e^\beta(F) \coloneqq \iint_{\R^6} \beta |e\cdot \nabla \log F|^2F \dd w \dd v  = \iint_{\R^6} \beta \frac{|e \cdot \nabla F|^2}{F} \dd w \dd v$$
for a given smooth vector field $e: \R^6 \to \R^6$ and smooth scalar function $\beta: \R^6 \to \R$. Additionally, we adopt from \cite{guillen2023global} the notation 
\begin{equation}
    L_b(F) \coloneqq b \cdot \nabla F \label{eq: L_b def}
\end{equation}
for a smooth vector field $b:\R^6 \to \R^6$. 
Taking $\beta =1$, it is observed in \cite{guillen2023global} that the following identity holds:
\begin{lem}[Essentially from Lemma $4.1$ of \cite{guillen2023global}] \label{l: Fisher direc}
    Let $e$ and $b$ be vector fields in $\R^6$, then
    \begin{equation}
        \Jap{I'_e(F),L_b(F)}= \iint_{\R^6} 2(e\cdot \grad \log F)([e,b]\cdot \grad \log F)F -\dv(b) (e\cdot \grad \log F)^2 F \dd w \dd v.
    \end{equation}
\end{lem}
The proof of this fact is contained in the proof of Lemma $4.1$ from \cite{guillen2023global}. Since we use Lemma \ref{l: Fisher direc} many times, often for a constant vector $e$, it is worthwhile to record it as its own fact. We may generalize this to 
\begin{lem} \label{l: Fisher direc scale}
    Let $e$ and $b$ be vector fields in $\R^6$ and $\beta$  a scalar function, then
    \begin{equation} 
        \Jap{(I_e^\beta)'(F),L_b(F)}= \iint_{\R^6} 2\beta (e\cdot \grad \log F)([e,b]\cdot \grad \log F)F -\dv(\beta b) (e\cdot \grad \log F)^2 F \dd w \dd v.
    \end{equation}
\end{lem}
\begin{proof}
    Again, this comes directly from the proof of Lemma $4.1$ of \cite{guillen2023global}. Direct computation yields
    \begin{align}
    \langle (I_e^\beta)'(F), L_b(F)\rangle
    &= \iint_{\R^6} 2\beta \frac{(e\cdot \nabla F)(e\cdot \nabla(b\cdot \nabla F))}{F} - \frac{(e\cdot \nabla F)^2}{F^2}(\beta b\cdot \nabla F) \dd w \dd v\\
    &= \iint_{\R^6} 2\beta \frac{(e\cdot \nabla F)(b\cdot \nabla(e\cdot \nabla F))}{F} - \frac{(e\cdot \nabla F)^2}{F^2}(\beta b \cdot \nabla F) + 2\beta \frac{(e\cdot \nabla F)([e,b]\cdot \nabla F)}{F} \dd w\dd v.
    \end{align} 
    All we have done is switched the order of differentiation at the expense of the commutator $[e,b]$. The first two terms write as
    \begin{align}
         2\beta \frac{(e\cdot \nabla F)(b\cdot \nabla(e\cdot \nabla F))}{F} - \frac{(e\cdot \nabla F)^2}{F^2}(\beta b \cdot \nabla F) 
         & = \beta b\cdot \nabla\left(\frac{(e\cdot \nabla F)^2}{F}\right) \\
         & = \dv\left(\frac{(e\cdot \nabla F)^2}{F} \beta b\right) - \dv(\beta b)\frac{(e\cdot \nabla F)^2}{F}
    \end{align}
    The first term disappears in the integration, and the lemma follows from $F \nabla \log F = \nabla F.$
\end{proof}
Recall from \cite{guillen2023global} we also have that
\begin{lem}[Essentially from Lemma $5.1$ of \cite{guillen2023global}] \label{l: Fisher}
    Let $b$ be a vector field in $\R^6$, then
    \begin{equation}
        \Jap{I'(F),L_b(F)} = \iint_{\R^6} 2  \Jap{Db \grad \log F, \grad \log F} F -\dv( b) |\grad \log F|^2 F \dd w \dd v.
    \end{equation}
\end{lem}

For $\beta: \R^6 \to \R$, it is also useful to define a weighted version of the full Fisher information
\begin{equation} 
\label{eq: FullFisherweighted}
    I^\beta(F) := \iint_{\R^6} \beta|\grad \log F|^2 F \dd w \dd v.
\end{equation}
Summing over the standard basis vectors $e_i \in \R^6$, Lemma \ref{l: Fisher direc scale} yields the following identity:
\begin{lem} \label{l: Fisher scale}
    Let $b$ be a vector field in $\R^6$ and $\beta$ a scalar function in $\R^6$, then
    \begin{equation}
        \Jap{(I^\beta)'(F),L_b(F)} = \iint_{\R^6} 2 \beta \Jap{Db \grad \log F, \grad \log F} F -\dv(\beta b) |\grad \log F|^2 F \dd w \dd v.
    \end{equation}
\end{lem}
Let us remark that all of the computations in this section and the two following are justified in the same way as in \cite{guillen2023global}. 
\section{The Maxwell Molecules Case}
In this section we focus on the Maxwell Molecules case $\alpha(\cdot) \equiv 1$ in \eqref{eq: KriegerLiftFinal}. As in \cite{guillen2023global}, proving that the Fisher information is decreasing along the Krieger-Strain flow is considerably simpler in the case of Maxwell Molecules.

Specializing to the case $\alpha(\cdot) \equiv 1$ in \eqref{eq: KriegerLiftFinal}, $Q_{KS}$ takes the form 
\begin{equation}
Q_{KS}  = Q_L + L_{\tilde{b}_0} \circ L_{\tilde{b}_0} + 6L_{\tilde{b}_0}. \label{eq: KriegerLiftMaxwell}
 \end{equation}
 where $\tilde{b}_0$ is as in \eqref{eq: b0def} and $L_{\tilde{b}_0}$ is as in \eqref{eq: L_b def}. Here we have used that $\dv(\tilde{b}_0)=6$ in dimension $d=3$. We will have to analyze the derivative of the Fisher information in the direction of each of the three operators on the righthand side. In particular, our proof uses that $\langle I'(F), Q_L(F)\rangle \le 0$, from \cite{guillen2023global}. 
 

 In our computations, it will be useful to define the following three vectors: 
 $$\nu_1 \coloneqq e_1+e_4, \ \ \nu_2 \coloneqq e_2+e_5, \ \ \nu_3 \coloneqq e_3+e_6,$$
with $e_i \in \R^6$ being the $i$'th standard basis vector.
 
We begin by studying the first order term of \eqref{eq: KriegerLiftMaxwell} with the following proposition. 

\begin{prop}
    \label{prop: MaxwellProp1}
    For any smooth positive function $F:\R^6 \to (0,\infty)$ with rapid decay at infinity we have
    \begin{equation}
    \langle I'(F), L_{\tilde{b}_0}(F)\rangle =-2I(F) - 2\sum_{i=1}^3 I_{\nu_i}(F). \label{eq: MaxwellLowerDerivative}
    \end{equation}
\end{prop}
\begin{proof}

By Lemma \ref{l: Fisher},
\begin{align} \langle I'(F), L_{\tilde{b}_0}(F)\rangle 
&  =  \iint_{\R^6} \frac{\nabla F \cdot (2D\tilde{b}_0 - \dv(\tilde{b}_0)\text{Id}_6)\nabla F}{F}\dd w \dd v \label{eq: maxwell2}
\end{align}
Computing directly,
\begin{align}
\frac12 \dv(\tilde{b}_0)\text{Id}_6 - D\tilde{b}_0 = \begin{pmatrix} 2\text{Id}_3 & \text{Id}_3 \\ \text{Id}_3 & 2\text{Id}_3\end{pmatrix}\ge 0.
\end{align}

Lastly we decompose the above matrix into $\text{Id}_6+ \sum_{i=1}^3 \nu_i\otimes \nu_i$, at which point \eqref{eq: maxwell2} reduces to \eqref{eq: MaxwellLowerDerivative} as claimed. 
\end{proof}
The decomposition according to the $\nu_i$ is particularly useful thanks to the observation that differentiation along $\nu_i$ commutes with differentiation along $\tilde{b}_0$, i.e., 
$$[\nu_i, \tilde{b}_0] =0, \ \ \text{ for each } 1\le i \le 3.$$
With this fact, we may apply Lemma \ref{l: Fisher direc} to immediately obtain
\begin{lem}
For any smooth positive function $F:\R^6 \to (0,\infty)$ with sufficient decay at infinity, there holds
\begin{align} \label{eq: maxwell3} \langle I_{\nu_i}'(F), L_{\tilde{b}_0}(F)\rangle =-6 I_{\nu_i}(F). \end{align}
\end{lem}
With these two lemmas in hand, we follow the approach of \cite{guillen2023global} and study the first order transport equation associated with $L_{\tilde{b}_0}$. For an arbitrary function $F_0= F_0(v,w)$, solve
\begin{equation}
\begin{cases}
\partial_tF(t,v,w) = L_{\tilde{b}_0}(F), \\
F(0,v,w) = F_0(v,w).
\end{cases} \label{eq: MaxwellTransport}
\end{equation}
The following lemma is immediate. 
\begin{lem}
\label{l: MaxwellLemma} 
    For any smooth, positive $F_0:\R^6 \to (0,\infty)$ with rapid decay at infinity, if $F$ solves \eqref{eq: MaxwellTransport}, then
    \begin{equation} \label{eq: Maxwell}
        \frac{d}{dt} I(F) = -2I(F) - 2\sum_{i=1}^3 I_{\nu_i}(F)
    \end{equation}
    and
    $$ \partial_{tt} F(0,v,w)=L_{\tilde{b}_0} \circ L_{\tilde{b}_0}(F_0)(v,w).$$
\end{lem}

Finally, we are able to prove the following special case of Theorem \ref{thm: FisherTheorem}.
\begin{prop}
    The Fisher information is monotone decreasing along the flow of the Krieger-Strain equation \eqref{eq: generalizedKrieger} when $\alpha \equiv 1.$
\end{prop}
\begin{proof}
    For an arbitrary smooth and positive $F_0:\R^6 \to (0,\infty)$ with sufficient decay, it suffices to prove 
    $$\langle I'(F_0), Q_{KS}(F_0)\rangle \le 0.$$
    With the decomposition of $Q_{KS}$ in \eqref{eq: KriegerLiftMaxwell}, and the powerful result 
    $$\langle I'(F_0), Q_L(F_0)\rangle \le 0$$
    from \cite{guillen2023global}, it suffices to prove that
    $$\langle I'(F_0), L_{\tilde{b}_0} \circ L_{\tilde{b}_0}(F_0)\rangle + 6\langle I'(F_0), L_{\tilde{b}_0}(F_0)\rangle \le 0.$$

    The former term must be computed; the latter term is computed in Proposition \ref{prop: MaxwellProp1}. To estimate the derivative of $I[\cdot]$ along $L_{\tilde{b}_0} \circ L_{\tilde{b}_0}$, we follow the approach of Guillen and Silvestre. Letting $F(t,v,w)$ solve \eqref{eq: MaxwellTransport}, we have on the one hand that
    \begin{align} \frac{d^2}{dt^2}I(F)\biggm|_{t=0} &= \langle I'(F_0), L_{\tilde{b}_0} \circ L_{\tilde{b}_0}(F_0)\rangle + \langle I''(F_0) L_{\tilde{b}_0}(F_0),  L_{\tilde{b}_0}(F_0)\rangle \ge \langle I'(F_0), L_{\tilde{b}_0} \circ L_{\tilde{b}_0}(F_0)\rangle
    \end{align}
    thanks to Lemma \ref{l: MaxwellLemma} and the convexity of $I[\cdot]$. On the other hand, we compute
    \begin{align}
        \frac{d^2}{dt^2}I(F) &= \frac{d}{dt} \langle I'(F), \partial_t F\rangle = \frac{d}{dt} \langle I'(F), L_{\tilde{b}_0}(F)\rangle\\
        &= \frac{d}{dt}\left(-2I(F) - 2\sum_{i=1}^3 I_{\nu_i}(F)\right)\\
        &= 4I(F) + 16 \sum_{i=1}^3 I_{\nu_i}(F),
    \end{align}
    by applying the identities \eqref{eq: maxwell3} and \eqref{eq: MaxwellLowerDerivative}. In particular, $\frac{d^2}{dt^2}I(F)\bigm|_{t=0} = 4I(F_0) + 16\sum_{i=1}^3I_{\nu_i}(F_0).$ Applying \eqref{eq: MaxwellLowerDerivative} again, it follows that
    \begin{align}
        \langle I'(F_0), L_{\tilde{b}_0} \circ L_{\tilde{b}_0}(F_0)\rangle + 6\langle I'(F_0), L_{\tilde{b}_0}(F_0)\rangle &\le 4I(F_0) + 16 \sum_{i=1}^3 I_{\nu_i}(F_0) + 6\left(-2I(F_0) - 2\sum_{i=1}^3 I_{\nu_i}(F_0)\right)\\
        &= -8I(F_0) + 4 \sum_{i=1}^3 I_{\nu_i}(F_0)\le0.
    \end{align}
    The last inequality follows simply from the fact that $\left\{\frac{\nu_i}{\sqrt{2}}\right\}$ is an orthonormal set in $\R^6$, which allows us to sum over the squares: 
$$\sum_{i=1}^3 |\nu_i \cdot \nabla F_0|^2 \le 2|\nabla F_0|^2,$$
which is enough to conclude the theorem. 
\end{proof}

We remark that in the Maxwell-Molecules case, we did not need to use any precise estimates from \cite{guillen2023global} for $\langle I'(F), Q_L(F)\rangle$. And despite the proof being more complicated than the proof of the corresponding Proposition $5.4$ of \cite{guillen2023global}, we obtain an even stronger decrease:
$$\langle I'(F), Q_{KS}(F)\rangle \le \langle I'(F), Q_L(F)\rangle \le 0.$$ This will not be the case for general $\alpha(\cdot)$, wherein worse estimates are obtained in the Krieger-Strain setting. The reason for this is that in the general case, one must contend with derivatives in the direction $n$ normal to the level sets of $\alpha(\cdot)$.

\section{The Fisher Information along the Krieger-Strain flow}
In this section we prove Theorem \ref{thm: FisherTheorem}. For a general radial interaction potential $\alpha(\cdot)$, it is convenient to introduce the notation
\begin{equation}
    \label{eq: Gamma_def}
     \Gamma = \Gamma(r; \alpha) \coloneqq \frac{r\alpha'(r)}{\alpha(r)} \text{ for }r>0. 
\end{equation}
If we take $\alpha(r)=r^\gamma$, then $\Gamma(r; \alpha) = \gamma$ identically. In the lifted setting, Theorem \ref{thm: FisherTheorem} may be reformulated as
\begin{prop}
\label{prop: Fisherprop}
    Let $F:\R^6 \to (0,\infty)$ be a smooth function with rapid decay at infinity so that $F(v,w)=F(w,v)$. Then there holds
    $$\langle I'(F), Q_{KS}(F)\rangle \le0$$
    for any interaction potential $\alpha(\cdot)$ satisfying 
    $$\Gamma = \Gamma(r; \alpha) \in [2-3\sqrt{3}, -2+2\sqrt{2}] \text{ for all } r>0.$$
\end{prop}

We now turn to showing that $\langle I'(F), Q_{KS}(F)\rangle \le 0$. We define $L_0 \coloneqq \sqrt{\alpha}\tilde{b}_0\cdot \nabla $ to simplify the notation. With this new notation, $Q_{KS}$ takes the form \eqref{eq: KriegerLiftFinal}
\begin{equation}
    \label{eq: KriegerLiftFinalRedux} Q_{KS} = Q_{L} + L_0 \circ L_0 + \left(6\sqrt{\alpha} + \frac{|v-w|\alpha'}{\sqrt{\alpha}}\right)L_0.
\end{equation}
We will have to analyze the derivative of the Fisher information along each of the three terms. As mentioned in the previous section, we will make use of more precise estimates for $\langle I'(F), Q_L(F)\rangle$. Throughout this section, we consider the lifted Krieger-Strain collision operator $Q_{KS}$  obtained in \eqref{eq: KriegerLiftFinalRedux} and the lifted Landau collision operator $Q_L$ obtained in \eqref{eq: expressionforQ_L} with the same interaction potential $\alpha(\cdot)$. The following proposition is one of the main inequalities from \cite{guillen2023global}. 

\begin{prop}
\label{prop: luisinequality}
     Let $F:\R^6 \to (0,\infty)$ be a smooth function with rapid decay at infinity so that $F(v,w)=F(w,v)$. Then with the $\tilde{b}_k$ as in Section $2$, the following inequality holds:
   \begin{equation}
       \label{eq: Luisinequality}
   \langle I'(F), Q_L(F)\rangle \le \iint_{\R^6}(\Gamma(r)^2-19)\frac{\alpha}{|v-w|^2}\sum_{k=1}^3 \frac{(\tilde{b}_k\cdot \nabla F)^2}{F} \dd w \dd v,
   \end{equation}
   where here $r=|v-w|$.
\end{prop}
This is a precise estimate for the decrease of the Fisher information along the flow of the Landau equation, for interaction potentials $\alpha(\cdot)$ with $|\Gamma(r,\alpha)| \le \sqrt{19}$ for all $r>0$. Unlike in the Maxwell Molecules case, it does not suffice for us to simply use $\langle I'(F), Q_L(F)\rangle \le 0$. This is not strong enough to compensate for the unfavorable terms $\langle I'(F), L_0\circ L_0F\rangle$ and $\langle I'(F), 6\sqrt{\alpha} + \frac{|v-w|\alpha'}{\sqrt{\alpha}}L_0(F)\rangle.$ 

Let us begin by analyzing the first derivative of the Fisher information along the flow of $L_0(F)$. In order to compactify notation, will be convenient to define the following two quantities:
\begin{equation}
\label{eq: beta_1}
    \beta_1 \coloneqq 6\sqrt{\alpha} + \frac{|v-w| \alpha'}{\sqrt{\alpha}} = \dv(\sqrt{\alpha}\tilde{b}_0),
\end{equation}
and
\begin{equation}
\label{eq: beta_2}
    \beta_2 \coloneqq 2\sqrt{\alpha} + \frac{|v-w| \alpha'}{\sqrt{\alpha}} = \beta_1 - 4\sqrt{\alpha}.
\end{equation}
These two scalar quantities arise naturally throughout the proof. We remind the reader that the main departure of the Krieger-Strain setting from the Landau setting is that one must contend with higher order derivatives in the direction $n$ normal to level sets of $\alpha(|v-w|)$, where 
$$n = \frac{1}{\sqrt{2}|v-w|} \tilde{b}_0.$$
The following lemma is an example of this. 
\begin{lem}
    \label{lem: Fisherlemma1}
    For any smooth positive function $F:\R^6 \to (0,\infty)$ with rapid decay, one has
    
    \begin{equation}
    \label{eq: FisherDeriv1} \Jap{I'(F),L_0(F)} = 2 \sum_{k=1}^3 I_{\tilde{b}_k}^{\sqrt{\alpha}/|v-w|^2}(F)+2I^{\beta_2}_n (F) -I^{\beta_1} (F) .
\end{equation}
\end{lem}
The quantities in \eqref{eq: FisherDeriv1} are defined in Section $2$.
\begin{proof}

First, from Lemma \ref{l: Fisher} we have that
\begin{align}
    \Jap{I'(F),L_0(F)}
    &=\iint_{\R^6} 2  \Jap{D(\sqrt{\alpha} \tilde{b}_0 )\grad \log F, \grad \log F} F -\dv(\sqrt{\alpha} \tilde{b}_0) |\grad \log F|^2 F \dd w \dd v.
\end{align}
Noting that $\dv(\sqrt{\alpha}\tilde{b}_0) = \beta_1$, we proceed to deal with the quadratic form $D(\sqrt{\alpha}\tilde{b}_0) \nabla \log F \cdot \nabla \log F$. The derivative of $D\tilde{b}_0$ may be written as 
\begin{equation}
    D\tilde{b}_0= \begin{pmatrix} \text{Id}_3 & -\text{Id}_3 \\ -\text{Id}_3 & \text{Id}_3\end{pmatrix}=\left( \sum_{k=0}^3\frac{1}{|v-w|^2}\tilde{b}_k\otimes \tilde{b}_k \right).
\end{equation}
 Further computation gives 
\begin{align}
    D(\sqrt{\alpha } \tilde{b}_0)&=\sqrt{\alpha} D\tilde{b}_0 + \frac{\alpha'}{\sqrt{2\alpha}} n \otimes \tilde{b}_0\\
    &=\sqrt{\alpha}\left( \sum_{k=0}^3\frac{1}{|v-w|^2}\tilde{b}_k\otimes \tilde{b}_k \right)+ \frac{|v-w|\alpha'}{\sqrt{\alpha}} n \otimes n\\
    &=\left( \sum_{k=1}^3\frac{\sqrt{\alpha}}{|v-w|^2}\tilde{b}_k\otimes \tilde{b}_k \right)+ \left(2\sqrt{\alpha}+ \frac{|v-w|\alpha'}{\sqrt{\alpha}} \right)n \otimes n\\
    &=\left( \sum_{k=1}^3\frac{\sqrt{\alpha}}{|v-w|^2}\tilde{b}_k\otimes \tilde{b}_k \right)+ \beta_2 n \otimes n.
\end{align}

The claimed identity \eqref{eq: FisherDeriv1} follows. 
\end{proof}
We now begin to deal with each term on the right-hand side individually. In computing $\langle I'(F), L_0 \circ L_0(F)\rangle$, we will need to compute the derivatives along $L_0$ of the three functionals on the righthand side of \eqref{eq: FisherDeriv1}. For this reason, we have the following lemma. 
\begin{lem}
    \label{lem: FisherDeriv2}
    For any smooth positive function $F:\R^6 \to \R$ with rapid decay, there holds
    \begin{align}
        \label{eq: Ib_kalphaderiv}
        \left\langle \left(I_{\tilde{b}_k}^{\sqrt{\alpha}/|v-w|^2}\right)'(F), L_0(F) \right\rangle&=\iint_{\R^6} -2\left( \frac{\alpha+ \alpha' |v-w|}{|v-w|^2}\right) \frac{(\tilde{b}_k\cdot \nabla F)^2}{F} \dd w \dd v,\\
        \label{eq: Inbeta2deriv}
        \left\langle (I_n^{\beta_2})'(F), L_0(F)\right\rangle &= \iint_{\R^6}( 2 \beta_2^2 -\dv(\beta_2 \sqrt{\alpha} \tilde{b}_0)) \frac{(n\cdot \nabla F)^2}{F} \dd w \dd v,\\
        \label{eq: Ibeta1deriv}
        \left\langle (I^{\beta_1})'(F), L_0(F)\right\rangle &= \iint_{\R^6} 2 \beta_1 \frac{\Jap{D(\sqrt{\alpha}\tilde{b}_0) \nabla F, \nabla F}}{F} -\dv (\beta_1 \sqrt{\alpha}\tilde{b}_0) \frac{|\nabla F|^2}{F} \dd w \dd v.
    \end{align}
\end{lem}
\begin{proof}
We begin with \eqref{eq: Ib_kalphaderiv}. From Lemma \ref{l: Fisher direc scale}, we have that for each $k=1,2,3$, the lefthand side of \eqref{eq: Ib_kalphaderiv} is equal to
\begin{align}
    &\iint_{\R^6} 2\frac{\sqrt{\alpha}}{|v-w|^2}(\tilde{b}_k\cdot \grad \log F)([\tilde{b}_k,\sqrt{\alpha}\tilde{b}_0]\cdot \grad \log F)F -\dv\left(\frac{\alpha}{|v-w|^2}\tilde{b}_0\right) (\tilde{b}_k\cdot \grad \log F)^2 F \dd w \dd v\\
    &=-\iint_{\R^6} \dv\left(\frac{\alpha}{|v-w|^2}\tilde{b}_0\right) (\tilde{b}_k\cdot \grad \log F)^2 F \dd w \dd v\\
    &=-\iint_{\R^6} \left( 6 \frac{\alpha}{|v-w|^2}+ 2 \frac{\alpha'}{|v-w|} -4 \frac{\alpha}{|v-w|^2} \right) (\tilde{b}_k\cdot \grad \log F)^2 F \dd w \dd v\\
    &=-2\iint_{\R^6} \left( \frac{\alpha+ \alpha' |v-w|}{|v-w|^2}\right) (\tilde{b}_k\cdot \grad \log F)^2 F \dd w \dd v,
\end{align}
as desired. Note that we used 
$$[\tilde{b}_k,\sqrt{\alpha}\tilde{b}_0] = \sqrt{\alpha}[\tilde{b}_k, \tilde{b}_0]=0 \ \ \ \text{ for } k=1,2,3,$$ 
where the first equality which is a simple consequence of the fact that $\tilde{b}_k \perp \nabla \alpha$, and the second is verified with \eqref{eq: commutatorfact} and direct computation. We record the following two commutator identities for $[n,\tilde{b}_0]$ and $[n,\sqrt{\alpha} \tilde{b}_0]$, with the latter following from the former:
\begin{equation}
    [n,\tilde{b}_0]=[n,\sqrt{2}|v-w|n]=\sqrt{2}(n\cdot \grad|v-w|) (n\cdot \grad)=2(n\cdot \grad),
\end{equation}
\begin{equation}
    [n,\sqrt{\alpha}\tilde{b}_0]=(n\cdot \grad \sqrt{\alpha})(\tilde{b}_0 \cdot \grad)+\sqrt{\alpha} [n,\tilde{b}_0]=\left(\frac{|v-w|\alpha'}{\sqrt{\alpha}} +2\sqrt{\alpha}\right)(n\cdot \grad)=\beta_2(n\cdot \grad).
\end{equation}
With the calculation $[n, \sqrt{\alpha}\tilde{b}_0] = \beta_2(n\cdot \nabla)$, we apply Lemma \ref{l: Fisher direc scale} to obtain the desired identity \eqref{eq: Inbeta2deriv} as follows:
\begin{align}
    \langle (I^{\beta_2}_n)'(F), L_0(F)\rangle
    &=\iint_{\R^6} 2 \beta_2 (n\cdot \grad \log F)([n,\sqrt{\alpha }\tilde{b}_0] \cdot \grad \log F)F - \dv (\beta_2 \sqrt{\alpha} \tilde{b}_0) (n\cdot \grad \log F)^2 F \dd w \dd v\\
    & =\iint_{\R^6}( 2 \beta_2^2 -\dv(\beta_2 \sqrt{\alpha} \tilde{b}_0)) (n\cdot \grad \log F)^2 F \dd w \dd v.
\end{align}
Moving onto \eqref{eq: Ibeta1deriv}, we use Lemma \ref{l: Fisher scale} to immediately obtain the final identity 
\begin{align}
    \langle (I^{\beta_1})'(F), L_0(F)\rangle
    &= \iint_{\R^6} 2 \beta_1 \Jap{D(\sqrt{\alpha}\tilde{b}_0) \grad \log F, \grad \log F} F -\dv (\beta_1 \sqrt{\alpha}\tilde{b}_0) |\grad \log F|^2 F \dd w \dd v.
\end{align}
\end{proof}
Now we continue with the proof of Proposition \ref{prop: Fisherprop}.
\begin{proof}[Proof of Proposition \ref{prop: Fisherprop}]
For any smooth positive function $F_0$ with rapid decay, $F_0=F_0(v,w)$, let us consider the initial value problem
\begin{equation}
    \begin{cases}
    \partial_t F(t,v,w)= L_0 (F),\\
    F(0,v,w)=F_0(v,w).
\end{cases}
\end{equation}
As in the Maxwell Molecules case, we have
\begin{align}
    \frac{d^2}{dt^2} I(F) 
    &= \frac{d}{dt} \Jap{ I'(F), L_0(F) }\\
    &=\Jap{I'(F), L_0\circ L_0 (F)}+\Jap{I''(F) L_0(F), L_0(F)}\\
    & \ge \Jap{I'(F), L_0\circ L_0 (F)}
\end{align}
owing to the convexity of the Fisher information. 
On the other hand, thanks to Lemma \ref{lem: Fisherlemma1} and Lemma \ref{lem: FisherDeriv2}, we may compute 
\begin{align}
 \frac{d^2}{dt^2}I(F) &= \frac{d}{dt}\Jap{ I'(F), L_0(F) }
     =\frac{d}{dt} \left( 2 \sum_{k=1}^3 I_{\tilde{b}_k}^{\sqrt{\alpha}/|v-w|^2}(F)+2I^{\beta_2}_n (F) -I^{\beta_1} (F)\right) \\
     &=\iint_{\R^6} \left( \frac{-4\alpha-4 \alpha' |v-w|}{|v-w|^2}\right) \sum_{k=1}^3\frac{(\tilde{b}_k\cdot \grad F)^2}{F} +( 4 \beta_2^2 -2\dv(\beta_2 \sqrt{\alpha} \tilde{b}_0)) \frac{(n\cdot \grad F)^2}{F}  \\
     & - 2 \beta_1 \frac{\Jap{D(\sqrt{\alpha}\tilde{b}_0) \grad F, \grad F}}{F} +\dv (\beta_1 \sqrt{\alpha}\tilde{b}_0) \frac{|\grad F|^2}{F} \dd w \dd v. \label{eq: FisherL0L0estimate}
\end{align}
At this point it is useful to rewrite $Q_{KS}$ in terms of $\beta_1$ as follows:
\begin{equation}
    Q_{KS}(F)=Q_L(F)+ L_0 \circ L_0 (F) +(6\alpha +|v-w| \alpha') (\tilde{b}_0 \cdot \grad F) = Q_L + L_0 \circ L_0 (F) + \sqrt{\alpha} \beta_1 (\tilde{b}_0 \cdot \grad F).
\end{equation}
To estimate $I'(F)$ along the flow of the first order operator $\sqrt{\alpha}\beta_1\tilde{b}_0 \cdot \nabla$, we apply Lemma \ref{l: Fisher} to find
\begin{equation}
    \Jap{ I'(F), L_{\sqrt{\alpha}\beta_1 \tilde{b}_0 }(F) }= \iint_{\R^6} 2 \frac{\Jap{D (\sqrt{\alpha} \beta_1 \tilde{b}_0) \grad F, \grad F}}{F} -\dv(\sqrt{\alpha} \beta_1 \tilde{b}_0) \frac{|\nabla F|^2}{F} \dd w \dd v.
    \label{eq: Fisherlowerorderterm}
\end{equation}
Using \eqref{eq: Fisherlowerorderterm} and the estimate on $\langle I'(F), L_0\circ L_0(F)\rangle$ from \eqref{eq: FisherL0L0estimate}, we obtain after some cancellation the estimate
\begin{align}
    \Jap{I'(F),Q_{KS}(F)-Q_L(F)}
     &\le \iint_{\R^6} \left( \frac{-4\alpha-4 \alpha' |v-w|}{|v-w|^2}\right) \sum_{k=1}^3\frac{(\tilde{b}_k\cdot \grad F)^2}{F} +( 4 \beta_2^2 -2\dv(\beta_2 \sqrt{\alpha} \tilde{b}_0)) \frac{(n\cdot \grad F)^2}{F} \\
    &+2 \frac{(\grad \beta_1 \cdot \grad F)(\sqrt{\alpha} \tilde{b}_0 \cdot \grad F)}{F} \dd w \dd v\\
    & =\iint_{\R^6} \left( \frac{-4\alpha-4 \alpha' |v-w|}{|v-w|^2}\right) \sum_{k=1}^3\frac{(\tilde{b}_k\cdot \grad F)^2}{F} +( 4 \beta_2^2 -2\dv(\beta_2 \sqrt{\alpha} \tilde{b}_0)) \frac{(n\cdot \grad F)^2}{F} \\
    &+4 |v-w| \sqrt{\alpha} \beta_1'  \frac{(n \cdot \grad F)^2}{F} \dd w \dd v\\
     & =\iint_{\R^6} \left( \frac{-4\alpha-4 \alpha' |v-w|}{|v-w|^2}\right) \sum_{k=1}^3\frac{(\tilde{b}_k\cdot \grad F)^2}{F} \\
     &+\left( 2 \frac{(|v-w|\alpha')^2}{\alpha}-8\alpha+8|v-w|\alpha' \right) \frac{(n\cdot \grad F)^2}{F} \dd w \dd v \label{eq: QKSminusQL}.\\
\end{align}
Let us briefly explain where the coefficient of $\displaystyle \frac{(n\cdot \nabla F)^2}{F}$ comes from. The strategy is to write $\beta_2= \beta_1-4\sqrt{\alpha}$ and to write everything in terms of $\beta_1$. Then one writes all gradients in terms of $n$, for given a function $\varphi:\R^6 \to \R$ which depends only on $|v-w|$, one has, with obvious abuse of notation, the identity $\nabla \varphi = \sqrt{2}\varphi' n$. With this in mind, it is easy to compute the derivatives, which we record plainly as follows:
\begin{align}\dv(\beta_2 \sqrt{\alpha}\tilde{b}_0) &= \dv(\beta_1 \sqrt{\alpha}\tilde{b}_0) - 4\dv(\alpha\tilde{b}_0)\\
&= \beta_1^2 + 2|v-w|\sqrt{\alpha}\beta_1' -24\alpha -8|v-w|\alpha',
\end{align}
and therefore
\begin{align}
    4\beta_2^2 - 2\dv(\beta_2 \sqrt{\alpha}\tilde{b}_0) + 4|v-w|\sqrt{\alpha}\beta_1'&=2\beta_1^2 -32\sqrt{\alpha}\beta_1 + 112\alpha + 16|v-w|\alpha'\\
    &=  2 \frac{(|v-w|\alpha')^2}{\alpha}-8\alpha+8|v-w|\alpha',
\end{align}
after using the definition of $\beta_1$ and simplifying. We remark that the $\beta_1'$ terms cancel out, which spares us from having to contend with second derivatives of $\alpha$. This is the only point where $\alpha''$ terms arise. We continue by writing \eqref{eq: QKSminusQL} in terms of $\Gamma=\Gamma(|v-w|;\alpha) = \displaystyle \frac{|v-w|\alpha'(|v-w|)}{\alpha(|v-w|)}$ with
\begin{equation}
    \Jap{I'(F),Q_{KS}(F)-Q_L(F)} \le  \iint_{\R^6} \left( \frac{-4\alpha(1+\Gamma)}{|v-w|^2}\right) \sum_{k=1}^3\frac{(\tilde{b}_k\cdot \grad F)^2}{F} +(2\Gamma^2+8\Gamma-8) \alpha \frac{(n\cdot \grad F)^2}{F} \dd w \dd v.
\end{equation}
If $2\Gamma^2+8\Gamma-8 \le 0$ for all $r>0$, or equivalently if 
\begin{equation}
    \Gamma= \Gamma(r) \in [-2-2\sqrt{2},-2+2\sqrt{2}] \text{ for all } r>0, \label{eq: firstGammacondition}
\end{equation} 
then it holds that
\begin{equation}
     \Jap{I'(F),Q_{KS}(F)-Q_L(F)} \le  \iint_{\R^6} \left( \frac{-4\alpha(1+\Gamma)}{|v-w|^2}\right) \sum_{k=1}^3\frac{(\tilde{b}_k\cdot \grad F)^2}{F} \dd w \dd v.
\end{equation}
The condition \eqref{eq: firstGammacondition} is the first requirement we impose upon $\Gamma$. Finally, we apply the result of Guillen and Silvestre from Proposition \ref{prop: luisinequality} to obtain 
\begin{equation}
    \Jap{I'(F),Q_{KS}(F)} \le \iint_{\R^6} (\Gamma^2-4\Gamma-4-19)  \frac{\alpha}{|v-w|^2} \sum_{k=1}^3 \frac{(\tilde{b}_k\cdot \grad F)^2}{F}\dd w \dd v.
\end{equation}
Therefore, if in addition to \eqref{eq: firstGammacondition}, we impose that $\Gamma^2 -4\Gamma -23\le 0$ for all $r>0$, equivalently
\begin{equation}
    \Gamma=\Gamma(r) \in [2-3\sqrt{3}, 2+3\sqrt{3}] \text{ for all } r>0, \label{eq: SecondGammacondition}
\end{equation}
we have the desired monotonicity
\begin{equation}
    \Jap{I'(F),Q_{KS}(F)} \le 0.
\end{equation}
Taking the intersection of the conditions \eqref{eq: firstGammacondition} and \eqref{eq: SecondGammacondition} we obtain the condition on $\Gamma$ as stated in Proposition \ref{prop: Fisherprop}. With the theory set forth in Section $2$, this also concludes the proof of Theorem \ref{thm: FisherTheorem}.
\end{proof}

As an immediate corollary of the preceding and the Sobolev embedding, we obtain the following bound on the $L^3$ norm:
\begin{cor}
\label{cor: L3estimate}
    Let $f:[0,T] \times \R^3 \to [0,\infty)$ be a classical solution to \eqref{eq: Krieger-Strain,divform2} such that $i(f_{\text{in}})<\infty$. Then 
    $$\|f(t)\|_{L^3} \le Ci(f_{\text{in}})$$
    for an absolute constant $C$. 
\end{cor}

We remark that from this point forward, we return to the original equation \eqref{eq: Kriegerpotential} on $\R^3$. With the monotonicity of the Fisher information $i(f)$ in hand, we make no more use of the lifted equation \eqref{eq: KriegerLiftFinal}.

\section{Propagation of Moments and an $L^\infty$ Estimate}
In this section, we establish growth estimates on all $s$-th moments of $f$, defined as 
\begin{equation}
    E_s(f) \coloneqq \int_{\R^3} f(v)|v|^s \dd s.
\end{equation}for $s \ge 0$, and as a consequence an unconditional a priori estimate on $\|f(t)\|_{L^\infty}$. The estimates rely on the quantity $\|f\|_{L^\infty([0,T]; L^3(\R^3))}$, which is controlled thanks to the monotonicity of the Fisher information and Corollary \ref{cor: L3estimate}. 
In this section and in what follows, we assume $\alpha(r)=r^\gamma$ for $\gamma \in [-3,-2)$. We will not state the range of $\gamma$ in every statement.

We also define the weighted Lebesgue norms $L^p_m$ and Sobolev norms $H^k_m$ of $f$ as follows:
\begin{equation}
    \norm{f}_{L^p_m}:=\norm{\Jap{v}^m f}_{L^p}, \quad 
    \norm{f}_{\dot{H}^{k}_m} \coloneqq \norm{D^k f}_{L^2_m}, \quad  \norm{f}_{H^{k}_m} \coloneqq \sum_{j=0}^k \norm{D^j f}_{L^2_m}.
\end{equation}
Note that $\norm{g}_{L^1_s}$ is not the same as $E_s(g)$, but is comparable to $E_s(g)+E_0(g)$. The mass and energy discussed in Section 2 correspond to $E_0(f(t))$ and $E_2(f(t))$, respectively.  Recall that the equation \eqref{eq: Kriegerpotential} conserves the mass, but not the energy. Even though the energy is not conserved, a linear growth estimate for the energy was obtained in \cite{gualdani2022hardy}, which is the following lemma.  

\begin{lem}[Essentially from \cite{gualdani2022hardy}]\label{l: energy} Let $f:[0,T] \times \R^3$ be a classical solution to \eqref{eq: Krieger-Strain,divform2}, with initial data $f_\text{in}$ having finite mass and energy. Then the energy $E_2(f(t))$ satisfies
$$E_2(f(t)) \le E_2(f_{\text{in}}) + tC_{\gamma} \|f\|_{L^\infty([0,t]; L^3)}^{1-\theta_0}\|f_{\text{in}}\|_{L^1}^{1+\theta_0} \le E_2(f_{\text{in}}) + TC_{\gamma} \|f\|_{L^\infty([0,T]; L^3)}^{1-\theta_0}\|f_{\text{in}}\|_{L^1}^{1+\theta_0}$$
where $\theta_0 \coloneqq \frac{|2+\gamma|}{2}$ for all $t \in [0,T]$.
\end{lem}
\begin{proof} The first inequality was obtained in \cite{gualdani2022hardy}, using the computation \eqref{eq: energygrowth}. The second inequality is trivial.
\end{proof}
\begin{remark}

    In \cite{gualdani2022hardy}, the linear growth of the energy was obtained for the range of $\gamma \in [-\frac52, -2]$, leaving open the range $[-3, -\frac 52)$. Their novel method was to apply Hardy's inequality to obtain the propagation of $L^p$ norm. In particular, they prove a propagation of $L^p$ norm for $1 \le p \le \frac{3+\gamma}{-2-\gamma}$, which necessitates $\gamma \ge -\frac52.$ And then to control the energy, they need an estimate on $\|f(t)\|_{L^p}$ for $p> \frac{3}{5+\gamma}$. Our propagation of $L^3$ norm from Corollary \ref{cor: L3estimate}, which does not depend on $\gamma$, is strong enough to obtain the propagation of energy for all $\gamma\in[-3,-2)$. 
\end{remark}

Before proving a growth estimate for all $s-$th moments, we need estimates of the upper and lower bounds for $a[f]$.
Interpolated with the conservation of $L^1$ norm, the control of $L^3$ norm gives us a control on $L^p$ norm for every $1 \le p \le 3$. A control on $L^p$ norm can be used to obtain an upper bound on the diffusion coefficient $a[f]=f \ast |\cdot|^{2+\gamma}$ by the following convolution inequality. The proof is contained in Remark $4.2$ of \cite{gualdani2022hardy}.

\begin{lem}\label{l: convoestimate}
    If $f \in L^p(\R^3)$ for $ p>\displaystyle \frac{3}{5+\gamma}$, it holds that
    $$
    \norm{a[f]}_{L^\infty} \le C_{p,\gamma}\|f\|_{L^p}^\theta ||f\|_{L^1}^{1-\theta}, \text{ where }\theta \coloneqq  \frac{|2+\gamma|p'}{3}, $$
    where $p'$ is the H\"older conjugate of $p$.
    In particular, for $p=3$, 
    \begin{equation}
        \norm{f \ast |\cdot|^{2+\gamma}}_{L^\infty } \le C_\gamma \norm{f}_{L^3}^{\theta_0} \norm{f}_{L^1}^{1-\theta_0}
    \end{equation}
    where $\theta_0:=\frac{|2+\gamma|}{2}$.
\end{lem}

 The lower bound for $a[f]$ is obtained in Corollary 4.3 of \cite{gualdani2022hardy} as a consequence of the inequality in our Lemma \ref{l: energy}.

\begin{lem}[Corollary 4.3 of \cite{gualdani2022hardy}] \label{l: abound}
    Let $f:[0,T]\times \R^3\to [0,\infty)$ be a classical solution to \eqref{eq: Kriegerpotential}. The diffusion coefficient $a[f]$ satisfies the bounds
    $$\label{eq: diffusionbounds} \ell \langle v\rangle^{2+\gamma} \le a[f](t,v) $$
    Here 
    $$\ell\coloneqq c_{\gamma} \frac{\|f_{\text{in}}\|_{L^1}^{-1-\gamma}}{\left(c_1 + c_2T \|f\|_{L^\infty([0,T]; L^3)}^{1-\theta_0} \|f_{\text{in}}\|_{L^1}^{1+\theta_0}\right)^{-2-\gamma}}$$
    where $\theta_0 =\frac{|2+\gamma|}{2}.$
\end{lem} 
\begin{proof}
     We can plug in $p=3$ to Corollary 4.3 of \cite{gualdani2022hardy} since $3>\frac{3}{5+\gamma}$.
\end{proof}
  The lower bound for the diffusion coefficient for the Landau equation was obtained in Lemma 3.2 of \cite{silvestre2017upper} by Silvestre, which uses the following idea. Heuristically, bounded energy prevents the $L^1$ mass from escaping off to infinity, and the bounded entropy prevents the mass from concentrating on a set of small measure. Since the energy is conserved for the Landau equation, it allows one to prove the following: there is a ball $B_R$, a threshold $c $, and an absolute constant $\mu$ for which $f \ge c$ on a subset of $B_R$ with measure at least $\mu$. Silvestre's approach leads to an ellipticity estimate that depends on the entropy, energy, and mass. The estimate we use from Lemma \ref{l: abound} does not rely on the entropy dissipation, but it does rely on the stronger $L^\infty_t L^3_v$ estimates coming from the monotonicity of the Fisher information. We opt to use this estimate essentially because it has already been done; our contribution towards Lemma \ref{l: abound} is only in proving that $f \in L^\infty_t L^p_v$ for some $p> \frac{3}{5+\gamma}$.

In the following proposition, we use the growth estimates of Lemma \ref{l: energy} to obtain estimates on all higher $L^1_s$ moments. This newly obtained control in $L^\infty([0,T];L^1_s)$ for any $s>0$ will be interpolated with the $L^\infty([0,T]; L^3)$ control of Corollary \ref{cor: L3estimate} to obtain an estimate in $L^\infty([0,T]; L^p_m)$ for any $1<p<3, m \ge 0$, which is an important step towards obtaining $L^\infty_{t,v}$ estimates. 

\begin{prop}\label{prop: momentsestimate}
    Let $f:[0,T] \times \R^3$ be a classical solution to \eqref{eq: Kriegerpotential} and suppose that $E_s(f_{\text{in}}) <\infty$ for some $s\ge 0.$ We have
    \begin{equation} \label{eq: momentestimate} E_s(f(t))\le C( E_{\max(s,2)}(f_{\text{in}})t^{(s-2)_+/2} + t^{s/2}). \end{equation}
    for some $C>0$ depending only on $s, \gamma, \norm{f_\text{in}}_{L^1}$, and $\norm{f}_{L^\infty([0,T]; L^3)}$. For $s\ge 2$, it simplifies to
    \begin{equation}
        E_s(f(t))\le C( E_{s}(f_{\text{in}})t^{(s-2)/2} + t^{s/2}).
    \end{equation}
    In particular, all moments stay bounded on finite time intervals, provided they are finite initially.
\end{prop}

\begin{proof}
 The proof is by induction on even $k \in \mathbf{N}$. For $k=0$, we know $E_0(f(t))$ is constant in time. For $k=2$, it follows from Lemma \ref{l: energy}. By interpolating $k=0$ and $k=2$ cases with H\"older's inequality, we conclude for all $k\in[0,2]$. Now we consider general $k>2$. 

We use the divergence form of the equation \eqref{eq: Krieger-Strain,divform2} and proceed as normally. As usual, we suppress the dependence on $t$. Differentiating and then integrating by parts twice, 
\begin{align}
\frac{d}{dt} E_k(f(t)) &= \int_{\R^3} |v|^k \partial_{v_i}\int_{\R^3} |v-w|^{2+\gamma} \partial_{v_i-w_i}(f(v)f(w)) \dd w \dd v\\
&= - k\iint_{\R^6} |v|^{k-2} |v-w|^{2+\gamma} v_i \partial_{v_i-w_i}(f(v)f(w)) \dd w \dd v\\
&= k \iint_{\R^6} \partial_{v_i-w_i}(|v|^{k-2} |v-w|^{2+\gamma} v_i ) f(v)f(w) \dd w \dd v.\end{align}

Here we have employed summation notation. Evaluating the derivative and summing over $i$,
we find
\begin{align} \frac{d}{dt}E_k(f(t)) &= k(k+1) \iint_{\R^6}  |v|^{k-2}|v-w|^{2+\gamma} f(v)f(w) \dd v\dd w + 2k(2+\gamma)\iint_{\R^6} v\cdot (v-w) |v|^{k-2} |v-w|^\gamma \\
&= k(k+1) \iint_{\R^6}  |v|^{k-2}|v-w|^{2+\gamma} f(v)f(w) \dd v\dd w \\
&+ k(2+\gamma)\iint_{\R^6} (|v|^{k-2}v-|w|^{k-2} w)\cdot (v-w) |v-w|^\gamma f(v)f(w) \dd v\dd w \\
& \coloneqq J_1(t) + J_2(t).
\end{align}

The latter expression for $J_2$ is obtained using symmetry. We only need to bound $J_1$ and $J_2$ from above. With Lemma \ref{l: convoestimate}, we estimate 
\begin{align} \label{eq: J1} J_1(t) &\le k(k+1) \|f(t)\ast |\cdot|^{2+\gamma}\|_{L^\infty} E_{k-2}(f(t))\\ 
&\le C_{k,\gamma}\|f(t)\|_{L^3}^{\theta_0} \norm{f(t)}_{L^1}^{1-\theta_0} E_{k-2}(f(t)).
\end{align}  
To handle $J_2$, note that $(|v|^{k-2}v-|w|^{k-2} w)\cdot (v-w) \ge 0$ for any $v, w\in \R^3$, owing to the convexity of the map $z \mapsto \frac1k |z|^k$. Then since the factor $2+\gamma \le 0$ for $\gamma \in [-3,-2]$, we simply obtain $J_2(t) \le 0$. With both of these in estimates, we simply apply Lemma \ref{l: convoestimate} to obtain
$$\frac{d}{dt} E_k(f(t)) \le C_{k,\gamma} \norm{f}_{L^\infty([0,T];L^3)}^{\theta_0} \norm{f_{\text{in}}}_{L^1}^{1-\theta_0} E_{k-2}(f(t)).$$
Hence, by the induction hypothesis,
\begin{align}
    E_k(f(t)) 
    &\le E_k(f_\text{in})+ C_{k,\gamma} t\norm{f}_{L^\infty([0,T];L^3)}^{\theta_0} \norm{f_{\text{in}}}_{L^1}^{1-\theta_0} E_{k-2}(f(t))\\
    &\le E_k(f_\text{in})+ C t( E_{\max(k-2,2)}(f_{\text{in}}) t^{(k-4)_+/2} + t^{(k-2)/2})\\
    & \le  C ( E_{k}(f_{\text{in}}) t^{(k-2)/2} + t^{k/2}).
\end{align}
Note that we used $E_{\max(k-2,2)}(f)\le E_k(f)+E_0(f)=E_k(f)+\norm{f_\text{in}}_{L^1}.$
\end{proof}

 We now turn to proving an $L^\infty$ estimate for solutions $f$ to $\eqref{eq: Kriegerpotential}$. Essentially, we will apply the following theorem for the equations in non-divergence form by Silvestre in \cite{silvestre2017upper}. We remark that an $L^\infty$ estimate is proven via Moser iteration for a particular range of $\gamma$ in \cite{gualdani2022hardy}. However, this estimate is not uniform in $(t,v)$, which makes it less immediately applicable for our purposes. But it should be stressed that we are able to apply the method of \cite{silvestre2017upper} owing in part to the linear growth of the energy, which originated in \cite{gualdani2022hardy}. 
 
\begin{thm}[Theorem 1.1 of \cite{silvestre2017upper}] \label{thm: upperbound}
Let $p\in [1,\infty)$, and assume $f:[0,1] \times \R^d \to \R$ be a function satisfying the following inequality in the classical sense:

    \begin{equation}
        \partial_t f \le a_{ij} \partial_{ij} f +C\norm{f}_{L^\infty}^{1+\alpha}.
    \end{equation}
where here $\alpha$ is a parameter in $[0, 2p/d)$, the coefficients $a_{ij}$ are locally uniformly elliptic, and for some constants $\delta>0, \Lambda >0, N \ge 0, \kappa \in \R$, and $\beta \ge -\kappa/d$, 
\begin{align}
    \det(a_{ij}) &\ge \delta \Jap{v}^{\beta d -\kappa}, \label{eq: luisellipticity}\\
    \{a_{ij}(t,v)\} &\le \Lambda \Jap{v}^{\min(2\beta, 2)}I, \label{eq: luisellipticityupper}\\
    \int_{\R^d} f(t,v)^p \Jap{v}^\kappa  \dd v&\le N \label{eq: luisLpk} \text{ for all } t\ge 0.
\end{align}
Then, for all $(t,v) \in [0,1] \times \R^d $, we have  
$$f(t,v) \le \begin{cases} Kt^{-d/(2p)} & t\le \tau, \\ K \tau^{-d/(2p)} & t\ge \tau.\end{cases}$$
Here $K$ and $\tau $ are constants depending on $\delta, \Lambda, N$ and the dimension $d$.
\end{thm}

See \cite[Theorem 2.1]{silvestre2017upper} for the explicit formulas for $K$ and $\tau$. The theorem is for functions defined on the time interval $[0,1]$, but we can easily rescale the function to change the time interval into $[0,T]$.

To apply Theorem \ref{thm: upperbound} to the equation \eqref{eq: Kriegerpotential}, we need ellipticity estimates for $a[f]$, an upper bound for $h[f]$, and an upper bound on $\norm{f}_{L^p_m}$ for appropriate $p,q$.  
The ellipticity estimate for $a[f]$ comes from Lemmas \ref{l: convoestimate} and \ref{l: abound}. An upper bound on  $\norm{f}_{L^p_m}$ comes from interpolating Proposition \ref{prop: momentsestimate} and Lemma \ref{l: convoestimate}.\\
It only remains to check the upper bound on $h[f]$:
\begin{lem}[Lemma 3.4 of \cite{silvestre2017upper}] \label{l: hbound}
    Let $f$ satisfy \eqref{eq: Kriegerpotential}. Then, for $d=3$, $h[f]$ satisfies for any $\gamma \in [-3,0]$ the estimate
    $$h[f](t,v) \le C_\gamma \|f_{\text{in}}\|_{L^1}^{1+\gamma/3}\|f(t)\|_{L^\infty}^{-\gamma/3}.$$
\end{lem}
This is taken verbatim from \cite{silvestre2017upper} and the proof is very standard. Now we can prove $L^\infty$ estimate for solutions $f$ to \eqref{eq: Kriegerpotential}.
\begin{thm}
\label{thm: kriegerstraininfinitybound}
Let $f:[0,T] \times \R^3 \to [0,\infty)$ be a solution of the equation \eqref{eq: Kriegerpotential}. Then there exist $K>0$ and $\tau >0$, depending only on $\gamma$, $T$, $\norm{f_\text{in}}_{L^1}$, $E_6(f_{\text{in}})$, and $\norm{f}_{L^\infty([0,T];L^3)}$,such that 
$$f(t,v) \le \begin{cases} Kt^{-d/(2p)} & t\le \tau, \\ K \tau^{-d/(2p)} & t\ge \tau.\end{cases}$$    
\end{thm}
\begin{proof}
    We apply Theorem \ref{thm: upperbound} for $d=3$, $\alpha=-\gamma/3$, $p=2$, $\beta=0 $, and $\kappa=-3(2+\gamma) \ge 0$. By Lemmas \ref{l: convoestimate} and \ref{l: abound}, we have
    \begin{align}
    \det(a[f]) &\ge l^3 \Jap{v}^{3(2+\gamma)}, \\
    a[f](t,v) &\le C_\gamma\norm{f}_{L^\infty([0,T];L^3)}^{\theta} \norm{f_\text{in}}_{L^1}^{1-\theta},
    \end{align}
    for  $\theta \coloneqq \frac{2|2+\gamma|}{3}.$ 
    Moreover, by interpolation inequality,
    \begin{align}
        \left(\int_{\R^3} f(t,v)^2 \Jap{v}^{-3(2+\gamma)}  \dd v\right)^2
        &\le \norm{f(t)}_{L^3}\norm{f(t)}_{L^1_{-6(2+\gamma)}}\\
        &\le C \|f\|_{L^\infty([0,T]; L^3)} (E_{\max(-6(2+\gamma),2)}(f_{\text{in}})t^{(-3(2+\gamma)-1)_+}+ t^{-3(2+\gamma)})\\
        &\le C \|f\|_{L^\infty([0,T]; L^3)} (E_{\max(-6(2+\gamma),2)}(f_{\text{in}})T^{(-3(2+\gamma)-1)_+}+  T^{-3(2+\gamma)})\\
        & \le \tilde{C}(T) (E_{\max(-6(2+\gamma),2)}(f_{\text{in}})+1)\\
        & \le \tilde{C}(T) (E_6(f_\text{in})+\norm{f_\text{in}}_{L^1}+1).
    \end{align}
\end{proof}
\begin{remark}
    We choose $p$ to be $2$ just for convenience. Any $\frac 32 < p <3$ works. If we choose $p$ sufficiently close to $3/2$, we can replace $E_6(f_\text{in})$ with $E_{4+\eps}(f_\text{in})$ for small $\eps>0$.
\end{remark}
Interpolating our results, we obtain higher moment estimates for all of the Lebesgue spaces $L^p(\R^3)$, $1<p<\infty.$
\begin{cor}
\label{cor: L^p_kestimate}
    Let $f$ be a solution of the equation \eqref{eq: Kriegerpotential}. Then for any $1\le p<\infty$ and $m \ge 0$, there holds
    $$\|f(t)\|_{L^p_m} = \|\Jap{v}^m f(t)\|_{L^p} \le \|f(t)\|_{L^\infty}^{1/p'}\norm{f(t)}_{L^1_{pm}}^{1/p}  \le C_{p,m,\gamma}
    \|f(t)\|_{L^\infty}^{1/p'}( \norm{f_\text{in}}_{L^1}+E_{pm}(f(t)) )^{1/p}.$$
    Here $p' = \frac{p}{p-1}$ is the usual H\"older conjugate.
\end{cor}

\begin{remark}
    As mentioned previously, the work of Gualdani and Guillen \cite{gualdani2022hardy} includes a proof of the conditional regularity result that if $f \in L^\infty((0,T);L^p(\R^3))$ for $p > \frac{3}{5+\gamma}$, then $f \in L^\infty_{loc}((0,T) \times \R^3)$. This would be enough to apply regularity theorems to conclude that $f \in C^\infty$ for as long as it exists. But the local estimate would not be enough to conclude $f \in L^\infty((0,T); L^p_m(\R^3))$ for arbitrary $m\ge 0$. This fact will be important in the following section, as we will we use $\|f(t)\|_{L^p_m}$ to control the growth of norms $\|f(t)\|_{H^k}$.
\end{remark}
\section{Propagation of $H^k$ Norms}
In this section, we establish an $L^\infty$ type continuation criterion for \eqref{eq: Kriegerpotential}. More precisely, we show that the growth of the unweighted Sobolev norm $\|f(t)\|_{H^k}$ is controlled by the $\|f(t)\|_{L^\infty}$ and $\|f(t)\|_{L^p_m}$, for sufficiently large $p,m>0$ and arbitrary $k\in \mathbb{N}$. By interpolation, the same is true of the $H^k_q$ norms.  So, in words, if $f(t)$ stays bounded in $L^\infty$ and $L^p_m$ on the time interval $[0,T]$, then $f$ may always be extended by the $H^k_q$ type continuation criteria of Section 7. Note that if the initial data has sufficient decay, then Corollary \ref{cor: L^p_kestimate} allows to control $L^p_m$ norms by the $L^\infty$ norm and the initial moments. We emphasize that the Fisher information is what allows us to actually prove global existence, for without Theorem \ref{thm: FisherTheorem} it is unknown how to control the $L^\infty$ and $L^p_m$ norms of $f(t)$. The following growth lemma is the main result of this section. 
\begin{prop} \label{l: H^klemma}
Let $f:[0,T] \times \R^3$ be a classical solution to \eqref{eq: Kriegerpotential} with initial data $f_{\text{in}}$ belonging to $H^k(\R^3)$ and $L^1_m(\R^3)$ for some $k \in \mathbb{N}$ and $m \in \mathbb{N}$ sufficiently large, depending on $k$ and $\gamma$. Suppose that
\begin{equation}
    \norm{f}_{L^\infty([0,T]\times \R^3)} \le K <\infty
\end{equation}
for some $K$.
Then for all $k \ge 0$, there exists $C_k$ depending only on  $\norm{f_\text{in}}_{L^1_m(\R^3)} $, $\norm{f_\text{in}}_{H^k(\R^3)}$, $T$ and $K$ such that 
  \begin{equation} \label{eq: inductionsobolev} \|f(t)\|_{H^{k}} \le C_k, \ \ 0 \le t \le T \end{equation}
  and
  \begin{equation} \label{eq: inductionsobolevweight} 
        \|F_k(t)\|_{L^2([0,T])}^2 \coloneqq \int_0^T\|D^{k+1}f (t)\|_{L^2_{(2+\gamma)/2}}^2 \dd t \le C_k.\end{equation}
\end{prop}
This proposition follows the strategy laid forth by Snelson in his recent work on the isotropic Boltzmann equation, see \cite{snelson2023global}. We remind the reader that the nondivergence form of the equation \eqref{eq: Kriegerpotential} is given by $\partial_t f = a[f]\Delta f -(2+\gamma)h[f]f$, where $a[f] = f\ast |\cdot|^{2+\gamma}$ and $\Delta a[f] = (2+\gamma)h[f]$. \\ We decompose the equation \eqref{eq: Kriegerpotential} \footnote{Typical notation for kinetic equations is $\partial_t f= \mathcal{Q}(f,f)$, where $\mathcal{Q}(f,g)$ is a nonlinear collision operator. When we decompose into $\mathcal{Q}=\mathcal{Q}_1 + \mathcal{Q}_2$, one should not confuse the capital $\mathcal{Q}_{KS}$ terms with the lifted Krieger-Strain collision operator of Sections $2-4$.} into 
\begin{equation}
    \label{eq: KriegerQ1Q2}
    \partial_t f= \mathcal{Q}_{KS}(f,f) = \mathcal{Q}_1(f,f) + \mathcal{Q}_2(f,f)
\end{equation}
where $\mathcal{Q}_1$ and $\mathcal{Q}_2$ are bilinear forms given by $$\mathcal{Q}_1(f,g) \coloneqq  a[f] \Delta g, \text{ and } \mathcal{Q}_2(f,g)\coloneqq -(2+\gamma)h[f] g.$$

It is natural that we consider $\mathcal{Q}_1$ and $\mathcal{Q}_2$ differently when estimating, as $\mathcal{Q}_1$ involves higher order differentiation and $\mathcal{Q}_2$ involves a more singular convolution kernel.  The proof of Proposition \ref{l: H^klemma} follows the general ideas of \cite{snelson2023global}. However, there are several key differences. The first is that there is a significant gap in the range of $\gamma$ considered:  in the case $d=3$, Snelson considers only $\gamma >-7/3$. As discussed in the introduction, Snelson's approach is based off of a propagation of $L^2$ norm made possible by a fractional Hardy's inequality. This method, which was initially introduced in \cite{gualdani2022hardy}, fails in our setting as Hardy's inequality becomes void as $\gamma \to -3$. Another departure from \cite{snelson2023global} is that our estimate avoids any dependence on weighted $L^\infty_m$ norms of $f$, instead opting to control the growth of the $H^k$ norm with the slightly weaker $L^\infty$ norm. 

Before proceeding with the proof of Proposition \ref{l: H^klemma}, we collect a few preliminary lemmas which are completely analogous to those in \cite{snelson2023global}. 
We begin with the following coercivity estimate. It is convenient to introduce the notation 
$$N_\gamma^f(g) \coloneqq \left(\int_{\R^3} a[f]|\nabla g|^2 \dd v\right)^{1/2}.$$

This term arises naturally in the proofs of local and global existence. The following is a straightforward adaptation of Lemma $5.1$ of \cite{snelson2023global}. 
\begin{lem}\label{l: coercivity}
Let $f:[0,T]\times \R^3$ be a classical solution to \eqref{eq: Kriegerpotential}. Let $g \in L^2(\R^3)$. Then there exists $C>0$ depending on $\norm{f_\text{in}}_{L^1(\R^3)}$ such that for all $t \in [0,T]$, 
\begin{equation}
    \int_{\R^3} g a[f(t)] \Delta g \dd v \le - N_\gamma^{f(t)}(g)^2 + C \|g\|_{L^2}^2 \|f(t)\|_{L^\infty}^{-\gamma/3}. \label{eq: coercivity}
\end{equation}
Moreover, if we assume
\begin{equation}
    \norm{f}_{L^\infty([0,T]\times \R^3)} \le K
\end{equation}
for some $K>0$, then there exists $c>0$ depending only on $\norm{f_\text{in}}_{L^1(\R^3)}$ and $K$ such that
\begin{equation}
    N_\gamma^{f(t)}(g) \ge c \|\nabla g\|_{L^2_{(2+\gamma)/2}(\R^d)}. \label{eq: Nlowerbd}
\end{equation}
    
\end{lem}
\begin{proof}
    For the first inequality, note that
    \begin{align} \int_{\R^3} g \nabla g \cdot \nabla a[f(t)]  \dd v &= \frac12 \int_{\R^3}  \nabla(g^2)\cdot  \nabla a[f(t)] \dd v\\
    &= -\frac{2+\gamma}{2} \int_{\R^3} h[f(t)] g^2 \dd v \end{align}
    which implies 
        \begin{align} \int_{\R^3} g a[f(t)] \Delta g \dd v &=- \int_{\R^3} |\nabla g|^2 a[f(t)] \dd v  - \int_{\R^3} g \nabla g \cdot \nabla a[f(t)] \dd v\\
    &\le  -N_\gamma^{f(t)}(g)^2 +C \|g\|_{L^2}^2\|f(t)\|_{L^\infty}^{-\gamma/3}
    \end{align}
    thanks to Lemma \ref{l: hbound}. Thus \eqref{eq: coercivity} is proven. 
  The lower bound \eqref{eq: Nlowerbd} is an immediate consequence of Lemma \ref{l: abound}.
\end{proof}

The following is a simple but useful estimate. 
\begin{lem}\label{l: derivative estimates a and h}
Let $f \in H^k_q(\R^3)$ for $q>\frac{7}{2} + \gamma$. For $|\beta| \le k$, 
\begin{equation}\label{ eq: aderivativeestimate 1}
    \|\partial^\beta a[f]\|_{L^\infty} \le C(\gamma,q) \|f\|_{\dot{H}^{|\beta|}_q} \le \|f\|_{H^{|\beta|}_q}.
\end{equation}
\end{lem}
\begin{proof}
    The proof is an immediate consequence of the commutation $\partial^\beta a[f] = a[\partial^\beta f]$ and Lemma \ref{l: app1}. 
\end{proof}

We may now turn to the proof of Proposition \ref{l: H^klemma}. Before beginning, we introduce the notation $A \lesssim B$ to denote that there exists a controlled quantity $C$ such that $A\le CB$. For instance, when we apply Lemma \ref{l: hbound} we write $\|h[f]\|_{L^\infty} \lesssim \|f\|_{L^\infty}^{-\gamma/3}$, where the  quantity $\|f\|_{L^1}^{1+\gamma/3}$ is constant by the conservation of mass, and absorbed into the $\lesssim$. This notation is intended to make the following series of estimates more readable.

\begin{proof}[Proof of Proposition \ref{l: H^klemma}.]
    We give a proof by induction. The goal is to prove that for all $k\ge 0$, there holds
\begin{equation}
    \label{eq: H^k estimate}
\frac{d}{dt}\int_{\R^3} |D^k f|^2 \dd v + c_0 \int_{\R^3} \Jap{v}^{2+\gamma} |D^{k+1}f|^2 \dd v \le C'_{k-1} \norm{f}_{H^k}^2(1+F_{k-1}(t)^2)
\end{equation}
for $c_0>0$ depending on $ \norm{f_\text{in}}_{L^1} $ and $K$,   $C'_{k-1}$ depending on $\norm{f_\text{in}}_{H^{k-1}}, \norm{f_\text{in}}_{L^1_m} $ ,and $K$. 

Before embarking on the proof, let us note that the conservation of mass interpolated with the assumption $f \in L^\infty([0,T] \times \R^3)$, guarantees that $f \in L^\infty([0,T]; L^3(\R^3))$. Thus we may apply Lemma \ref{l: abound} to obtain a lower bound on the ellipticity $a[f](t,v) \ge \ell \Jap{v}^{2+\gamma}$ on $[0,T] \times \R^3$, with $\ell>0$ defined precisely in that lemma. Ultimately in the proof we shall need $f \in L^\infty([0,T]; L^p_q)$ for $p,q$ large enough (for simplicity, we will take $p=4$). Taking $m$ large enough as in the statement of the Proposition, this will always be possible by the propagation of $L^1_s$ moments from Proposition \ref{prop: momentsestimate}, the assumption $\|f\|_{L^\infty([0,T]\times \R^3)} \le K$, and Corollary \ref{cor: L^p_kestimate}. So, the $L^\infty$ assumption is what allows for much of the control on the equation and the solution. 

We begin with the base case $k=0$. By the interpolation reasoning of the previous paragraph, the $L^2$ bound \eqref{eq: inductionsobolev} is satisfied for $k=0$. However, we still need to control the coercivity term $\int_{\R^3} \Jap{v}^{2+\gamma}|\nabla f|^2$. To that end, 
multiply \eqref{eq: Kriegerpotential} by $f$ and integrate to get
\begin{align} \frac12 \frac{d}{dt} \int_{\R^3} f^2 \dd v &= \int_{\R^3} f a[f] \Delta f \dd v -(2+\gamma) \int_{\R^3} h[f]f^2 \dd v.
\end{align}
The first integral integrates by parts to 
$$ \int_{\R^3} f a[f] \Delta f \dd v = -\int_{\R^3}  a[f]|\nabla f|^2 \dd v  - \frac12 \int_{\R^3}  \nabla f^2 \cdot \nabla a[f] \dd v$$
$$= -\int_{\R^3} a[f] |\nabla f|^2 \dd v + \frac{2+\gamma}{2} \int_{\R^3} h[f]f^2 \dd v.$$

We obtain
$$\frac12 \frac{d}{dt} \int_{\R^3} f^2 \dd v + \int_{\R^3} a[f]|\nabla f|^2 \dd v = -\frac{2+\gamma}{2} \int_{\R^3} h[f] f^2 \dd v$$
and conclude with the lower bound on $a[f]$ from Lemma \ref{l: abound} and the upper bound on $h[f]$ from Lemma \ref{l: hbound}. Note that $\norm{f}_{L^\infty([0,T];L^3(\R^3))}$ is controlled by $K$ and $\norm{f_\text{in}}_{L^1}$, hence $l$ from Lemma \ref{l: abound} as well.

Suppose that \eqref{eq: inductionsobolev} and \eqref{eq: inductionsobolevweight} hold for all $j\le k-1,$ where $k \ge 1$. Let $\beta =(\beta_1, \ldots \beta_d)$ be a multi-index with $|\beta|=k$. First, differentiate the equation \eqref{eq: Kriegerpotential} by $\partial^\beta$ and then multiply both sides by $\partial^\beta f$. Then sum over multi-indices $|\beta|=k$ and integrate to obtain 
    \begin{equation}
    \label{eq: inductionderivative} \frac12 \frac{d}{dt} \int_{\R^3} |D^kf|^2 \dd v = \sum_{\beta_1 + \beta_2=\beta, |\beta|=k} \left(\int_{\R^3} \partial^\beta f \mathcal{Q}_1(\partial^{\beta_1}f, \partial^{\beta_2}f) \dd v + \int_{\R^3} \partial^\beta f\mathcal{Q}_2(\partial^{\beta_1}f, \partial^{\beta_2}f) \dd v \right) 
    \end{equation}

which follows from the product rule $\partial_i \mathcal{Q}(F,G) = \mathcal{Q}(\partial_i F, G) + \mathcal{Q}(F, \partial_i G)$. 
We will have to estimate the terms of \eqref{eq: inductionderivative} in many different cases, depending on where the derivatives fall. In the following analysis we suppress the dependence of $f$ on $t$, writing simply $f$ to denote $f(t,\cdot)$. 
\\\par
\textbf{Case One: Estimating the $\mathcal{Q}_1$ term when $|\beta_1|=0$.
} In this case, we have 
    \begin{align} \label{eq: Q1estimatesecondslot}
    \int_{\R^3} \bdary^\beta f \mathcal{Q}_1(f,\bdary^\beta f) \dd v& = \int_{\R^3} \bdary^\beta f a[f] \Delta \bdary^\beta f \dd v \le -N_\gamma^f(\partial^\beta f)^2 +C \|f\|_{H^k}^2 \|f\|_{L^\infty}^{-\gamma/3}\end{align}
    where the constant $C$ depends only on $\gamma$ and the hydrodynamic quantities of $f$.  
    The  inequality \eqref{eq: Q1estimatesecondslot} is precisely the first statement of Lemma \ref{l: coercivity}.
\\\par 
\textbf{Case Two: Estimating the $\mathcal{Q}_1$ terms when $|\beta_1|=1.$}  In this case, we prove that for all $\eps>0$ and for $q$ such that $ q>\frac{9}{4}$, there holds

\begin{equation}
    \label{eq: Q1 beta2=k-1}
    \int_{\R^3} \partial^\beta f  \mathcal{Q}_1(\bdary^{\beta_1} f, \bdary^{\beta_2} f) \dd v \lesssim \frac1\eps \|f\|_{H^k}^2 \|f\|_{L^4_q}^2+ \eps \|f\|_{\dot{H}^{k+1}_{(2+\gamma)/2}}^2.
\end{equation}
in keeping with the notation of \eqref{eq: inductionsobolevweight}. 
Now, since $\partial^{\beta_1}$ is a first order derivative, it holds that $|\partial^{\beta_1} a[f]| \lesssim_\gamma f \ast |\cdot|^{1+\gamma}$. Applying Lemma \ref{l: app4} for $p=4$, we obtain 

\begin{align}
    \int_{\R^3} \bdary^\beta f  \bdary^{\beta_1} a[f] \lap \bdary^{\beta_2}f \dd v &= \int_{\R^3} \bdary^\beta f \left[ \Jap{v}^{(-2-\gamma)/2} \partial^{\beta_1} a[f]\right]\left[\Jap{v}^{(2+\gamma)/2} \Delta \bdary^{\beta_2}f\right]\dd v\\ 
    & \lesssim \|f\|_{H^k}\|D^{k+1}f\|_{L^2_{(2+\gamma)/2}} \|f \ast |\cdot|^{\gamma+1}\|_{L^\infty_{(-2-\gamma)/2}}\\ 
    &\lesssim \|f\|_{H^k} \|f\|_{\dot{H}^{k+1}_{(2+\gamma)/2}} \|f\|_{L^4_q} \lesssim \frac1\eps \|f\|_{H^k}^2 \|f\|_{L^4_q}^2+ \eps \|f\|_{\dot{H}^{k+1}_{(2+\gamma)/2}}^2.
\end{align}
for $q>\frac{9}{4}.$ Note that, since $\gamma \in [-3,-2)$, it holds that $1+\gamma>-\frac 94$ and also that $\frac{-2-\gamma}{2} \le -\frac{4}{3}(\gamma+1)$. These inequalities ensure that we may apply Lemma \ref{l: app4} with $p=4$ and $\mu=1+\gamma$. Thus \eqref{eq: Q1 beta2=k-1} is proven.  
\\\par 
\textbf{Case Three: Estimating the $\mathcal{Q}_1$ terms when $|\beta_1|=2$.} This term only appears for $k \ge 2$, and so we can assume $k \ge 2$. In this case, we obtain the estimate 
\begin{equation}
    \label{eq: Q1 beta2=k-2}
    \int_{\R^3} \bdary^\beta f  \mathcal{Q}_1(\bdary^{\beta_1} f, \bdary^{\beta_2} f) \dd v \lesssim  \frac {1}{\eps} \|f\|_{H^k}^2 \left( \norm{f}_{\dot{H}^{k}_{(2+\gamma)/2}}^2+ \norm{f}^2_{L^2_\theta}\right) + \eps \norm{f}_{\dot{H}^{k+1}_{(2+\gamma)/2}}^2.
\end{equation}
We use Lemma \ref{l: app4} with $p=2$ and Lemma \ref{l: app7} to get
\begin{equation}
    \|\partial_{ij}a[f]\|_{L^\infty_{(-2-\gamma)/2}} =\norm{a[\partial_{ij}f]}_{L^\infty_{(-2-\gamma)/2}}\lesssim \norm{\partial_{ij} f }_{L^2_q} \le \norm{f}_{\dot{H}^2_q} \lesssim \norm{f}_{\dot{H}^{k+1}_{(2+\gamma)/2}} +\norm{f}_{L^2_\theta}
\end{equation}
for $q>\frac{3}{2}$ and $\theta$ large enough. Note that $2+\gamma>-\frac{3}{2}$ and $\frac{-2-\gamma}{2} \le -2(2+\gamma)$, so Lemma \ref{l: app4} applies.
It is then immediate that
\begin{align}
    \int_{\R^3} \bdary^\beta f  \bdary^{\beta_1}a[f] \Delta \bdary^{\beta_2} f \dd v
    &\le \|f\|_{\dot{H}^k_{(2+\gamma)/2}}\norm{f}_{H^k}\|\bdary^{\beta_1}a[f]\|_{L^\infty_{(-2-\gamma)/2}}\\
    &\lesssim \norm{f}_{\dot{H}^{k}_{(2+\gamma)/2}} \norm{f}_{H^k} \left( \norm{f}_{\dot{H}^{k+1}_{(2+\gamma)/2}} +\norm{f}_{L^2_\theta} \right)\\
    & \lesssim \frac {1}{\eps} \|f\|_{H^k}^2 \left(  \norm{f}_{\dot{H}^{k}_{(2+\gamma)/2}}^2+ \norm{f}^2_{L^2_\theta}\right) + \eps \norm{f}_{\dot{H}^{k+1}_{(2+\gamma)/2}}^2
\end{align}
and \eqref{eq: Q1 beta2=k-2} follows as claimed. Note that we have used $\norm{f}_{\dot{H}^{k}_{(2+\gamma)/2}} \le \norm{f}_{H^k}.$
\\\par 
\textbf{Case Four: Estimating the $\mathcal{Q}_1$ terms when $3\le |\beta_1|\le k.$} In this case, we have the estimate 
\begin{equation}
    \label{eq: Q_1intermediaterange} 
    \int_{\R^3} \bdary^\beta f \mathcal{Q}_1(\bdary^{\beta_1}f, \bdary^{\beta_2}f)\dd v \lesssim \frac1\eps \|f\|_{H^k}^2 \left(1+ \norm{f}^2_{L^2_\theta}\right) + \eps \|f\|_{\dot{H}^{k+1}_{(2+\gamma)/2}}^2.
\end{equation}

To prove \eqref{eq: Q_1intermediaterange}, observe that
\begin{align}
    \int_{\R^3} \bdary^\beta f \mathcal{Q}_1(\bdary^{\beta_1}f, \bdary^{\beta_2}f) \dd v
    & = \int_{\R^3} \bdary^\beta f \bdary^{\beta_1}a[f] \Delta \bdary^{\beta_2}f \dd v\\
    &\le \|f\|_{H^k}\|f\|_{H^{2+|\beta_2|}}\|\partial^{\beta_1}a[f]\|_{L^\infty}. 
\end{align}
Applying Lemmas \ref{l: derivative estimates a and h} and \ref{l: app7}, we may bound $$\|\partial^{\beta_1}a[f]\|_{L^\infty} \lesssim \|f\|_{H^{k}_q} \lesssim \|f\|_{\dot{H}^{k+1}_{(2+\gamma)/2}}+\|f\|_{L^2_\theta}  $$ for some $q>\frac{7}{2}+\gamma $ and $\theta$ large enough. As $|\beta_2|+2 \le k-1$ by assumption, we arrive at the estimate
$$  \int_{\R^3} \bdary^\beta f \bdary^{\beta_1}a[f] \Delta \bdary^{\beta_2}f \dd v
\lesssim \|f\|_{H^k}\|f\|_{H^{k-1}} (\|f\|_{\dot{H}^{k+1}_{(2+\gamma)/2}}+\|f\|_{L^2_\theta} ) \lesssim  \frac1\eps \|f\|_{H^k}^2\left(1+ \norm{f}^2_{L^2_s}\right) + \eps \|f\|_{\dot{H}^{k+1}_{(2+\gamma)/2}}^2
$$
Here we have absorbed into the $\lesssim$ the induction hypothesis that $\|f(t)\|_{H^{k-1}}\le C_{k-1}$.  
\\\par
Now we move onto the $\mathcal{Q}_2$ terms, which must be handled with more involved convolution estimates. 

\textbf{Case Five: Estimating the $\mathcal{Q}_2$ terms when $|\beta_1|=k$ or $|\beta_1|=0$.} In these two cases, for $\theta>3/2+\gamma$ we have for all $\eps>0$ the estimates
\begin{equation}
    \label{eq: Q_2 edgecase1} \int_{\R^3} \partial^\beta f \mathcal{Q}_2(\bdary^\beta f, f)\dd v = -(2+\gamma)  \int_{\R^3} \partial^\beta f \partial^\beta h[f] f \dd v \lesssim \frac1\eps \|f\|_{H^k}^2 \left(1+\|f\|_{L^\infty}^2+ \|f\|_{L^2_\theta}^2 \right) + \eps \|f\|_{\dot{H}^{k+1}_{(2+\gamma)/2}}^2
\end{equation}
as well as
\begin{equation}
    \label{eq: Q_2 edgecase2} 
    \int_{\R^3} \partial^\beta f \mathcal{Q}_2(f, \bdary^\beta f) \dd v
    =-(2+\gamma) \int_{\R^3} \partial^\beta f h[f] \partial^\beta f \dd v \lesssim \|f\|_{L^\infty}^{-\gamma/3}\|f\|_{H^k}^2.
\end{equation}

The second estimate \eqref{eq: Q_2 edgecase2} is immediate from Lemma \ref{l: hbound}. In the case $\gamma=-3$, the \eqref{eq: Q_2 edgecase1} is identical to \eqref{eq: Q_2 edgecase2} and we get the same estimate. 
In the case $-3< \gamma < -2$, we instead use Lemma \ref{l: app3} to get
   \begin{align}
    \int_{\R^3} \bdary^\beta f \bdary^{\beta} h[f]  f \dd v 
    &\lesssim  \|\partial^\beta f\|_{L^2} \|f\|_{L^\infty} \|\bdary^\beta f \ast |\cdot|^\gamma\|_{L^2}\\
     &\lesssim  \|f\|_{H^k}\|f\|_{L^\infty} \|f\|_{H^k_q} \lesssim \|f\|_{H^k}\|f\|_{L^\infty} (\|f\|_{\dot{H}^{k+1}_{(2+\gamma)/2}} + \|f\|_{L^2_\theta})\\
     &\lesssim  \frac1\eps \|f\|_{H^k}^2 \left(1+\|f\|_{L^\infty}^2+ \|f\|_{L^2_\theta}^2 \right) + \eps \|f\|_{\dot{H}^{k+1}_{(2+\gamma)/2}}^2.
\end{align}
It is clear that we can choose $q,\theta$ large enough so that the two estimates agree, and in all cases \eqref{eq: Q_2 edgecase1} is proven. 
\\\par 
\textbf{Case Six: Estimating the $\mathcal{Q}_2$ terms when $1\le |\beta_1| \le k-2$.} In this case, for all $\eps>0$ we have the estimate 
\begin{equation}
\label{eq: Q_2intermediate}
 \int_{\R^3} \partial^\beta f \mathcal{Q}_2(\partial^{\beta_1}f, \partial^{\beta_2}f)\dd v
 =-(2+\gamma) \int_{\R^3} \partial^\beta f \partial^{\beta_1}h[f] \partial^{\beta_2} f \dd v \lesssim  \frac1\eps \|f\|_{H^k}^2 \left(1+\|f\|_{L^2_\theta}^2 \right) + \eps \|f\|_{\dot{H}^{k+1}_{(2+\gamma)/2}}^2  
\end{equation}
In the Coulomb case $\gamma=-3$, we have the particularly simple estimate 
\begin{align} \int_{\R^3} \partial^\beta f \partial^{\beta_1} h[f] \partial^{\beta_2} f \dd v &= 4\pi
\int_{\R^3} \partial^\beta f \partial^{\beta_1}f \partial^{\beta_2}f \dd v
\lesssim \|f\|_{H^k} \|\partial^{\beta_1} f\|_{L^\infty} \|\partial^{\beta_2}f\|_{L^2}\\
&\lesssim \|f\|_{H^k} \|f\|_{H^{2+|\beta_1|}}\|f\|_{H^{k-1}}  \lesssim C_{k-1}\|f\|_{H^k}^2,
\end{align}
using the Sobolev embedding $H^2(\R^3) \hookrightarrow L^\infty(\R^3)$, as well as that $|\beta_1| \le k-2$ and $|\beta_2|\le k-1$. In particular, we use the induction hypothesis to control the $H^{k-1}$ norm. Now, for $-3 <\gamma < -2$, notice that $\partial^{\beta_1} h[f] = c_{\gamma} \partial^{\beta_1} \Delta a[f]$ and apply Lemma \ref{l: derivative estimates a and h} with $q>\frac{7}{2}+\gamma$ to obtain
\begin{align}
 \int_{\R^3} \partial^\beta f \partial^{\beta_1} h[f] \partial^{\beta_2}f \dd v&= 
\int_{\R^3} \partial^\beta f \partial^{\beta_1}h[f] \partial^{\beta_2}f \dd v\\ 
&\lesssim \|f\|_{H^k} \|f\|_{H^{k-1}} \|\partial^{\beta_1}\Delta a[f]\|_{L^\infty} \lesssim \|f\|_{H^k} \|f\|_{H^{k-1}} \|f\|_{H^{2+|\beta_1|}_{q}}\\
&\lesssim  \frac1\eps \|f\|_{H^k}^2 \left(1+\|f\|_{L^2_\theta}^2 \right) + \eps \|f\|_{\dot{H}^{k+1}_{(2+\gamma)/2}}^2 
\end{align}
where the last line follows from the induction hypothesis $\|f(t)\|_{H^{k-1}} \le C_{k-1}$ and Lemma \ref{l: app7} with $\theta$ taken large enough. Thus \eqref{eq: Q_2intermediate} is proven. 
\\\par 
\textbf{Case Seven: Estimating the $\mathcal{Q}_2$ terms when $|\beta_1|=k-1$.} This case does not overlap with Case Five only if $k \ge 2$, so we can assume $k\ge 2$. We estimate differently for $k=2$ and $k\ge 3$. In the former case, we prove that there exists an $\theta\in \mathbb{N}$ such that for all $\eps>0$, the following inequality holds:
\begin{equation}
    \label{eq: Q2 beta_2=1 k=2 }
      \int_{\R^3} \partial^\beta f \mathcal{Q}_2(\partial^{\beta_1}f, \partial^{\beta_2}f)\dd v
\lesssim \|f\|_{L^\infty}^{-\gamma/3} \|f\|_{H^2}^2 + \eps\|f\|_{\dot{H}^3_{(2+\gamma)/2}}^2 + \frac1\eps \|f\|_{L^\infty}^{-2\gamma/3}\left(\|f\|_{H^2}^2 + \|f\|_{L^2_\theta}^2\right).
\end{equation}

For $k\ge 3$, there exists an $\theta \in \mathbb{N}$ such that the following inequality holds:
\begin{equation}
\label{eq: Q2 beta_2=1}
 \int_{\R^3} \partial^\beta f \mathcal{Q}_2(\partial^{\beta_1}f, \partial^{\beta_2}f)\dd v
\lesssim  \|f\|_{H^k}^2\left(\|f\|_{\dot{H}^k_{(2+\gamma)/2}} + \|f\|_{L^2_\theta}\right) 
\end{equation}
Let us handle the latter case $k\ge3$ first. To do so, we split into two sub-cases. For $-3<\gamma<-2$, we estimate as follows:
\begin{align}
    \int_{\R^3} \bdary^\beta f \bdary^{\beta_1}h[f] \bdary^{\beta_2}f \dd v & \le \|f\|_{H^k} \|\nabla f\|_{L^\infty} \|\partial^{\beta_1} f \ast|\cdot|^\gamma\|_{L^2}\\ 
    &\lesssim \|f\|_{H^k} \|f\|_{H^3}\|\partial^{\beta_1}f\|_{L^2_q} \lesssim \|f\|_{H^k} \|f\|_{H^3}\|f\|_{\dot{H}^{k-1}_{q}}\\
    &\lesssim \|f\|_{H^k}\|f\|_{H^3}\left(\|f\|_{\dot{H}^k_{(2+\gamma)/2}} + \|f\|_{L^2_\theta}\right),
\end{align}
and \eqref{eq: Q2 beta_2=1} follows because $k\ge 3$. Note we used $|\beta_2|=1$, and that to control $\|\nabla f\|_{L^\infty}$ by $\|f\|_{H^3}$, we used the Sobolev embedding $H^2(\R^3)\hookrightarrow L^\infty(\R^3)$. To control $\|\partial^{\beta_1} f \ast|\cdot|^\gamma\|_{L^2}$ by $\|f\|_{\dot{H}^{k-1}_\theta}$ for $q>\gamma+3$, we used Lemma \ref{l: app3}. Finally the last line is an application of Lemma \ref{l: app7}, taking $\theta$ as large as needed. The Coulomb case $\gamma=-3$ is even simpler; the only difference being that one does not need to appeal to Lemma \ref{l: app3}. In all cases, the estimate \eqref{eq: Q2 beta_2=1} is proven.

In the case $k=2$, it is convenient to write $\partial^{\beta_1} = \partial_i$, $\partial^{\beta_2} = \partial_j$, for arbitrary indices $1\le i, j \le 3$. Integrating by parts, we estimate  
\begin{align} \int_{\R^3} \bdary^\beta f Q_2(\bdary^{\beta_1}f, \bdary^{\beta_2}f) \dd v &= -(2+\gamma) \int_{\R^3} \partial_i h[f] \partial_{ij}f \partial_j f \dd v= -\frac{2+\gamma}{2} \int_{\R^3} \partial_i h[f] \partial_i (\partial_j f)^2 \dd v\\
    &= (2+\gamma) \int_{\R^3} h[f]\left((\partial_{ij}f)^2 + \partial_{jii}f \partial_jf\right)\dd v\\
    &\lesssim_\gamma \|f\|_{L^\infty}^{-\gamma/3} \|f\|_{H^2}^2 + \eps \int_{\R^3} \Jap{v}^{2+\gamma}|D^3f|^2\dd v + \frac1\eps  \int_{\R^3} h[f]^2 \Jap{v}^{-2-\gamma} |\nabla f|^2\dd v\\
    &\lesssim\|f\|_{L^\infty}^{-\gamma/3} \|f\|_{H^2}^2 + \eps\|f\|_{\dot{H}^3_{(2+\gamma)/2}}^2 + \frac1\eps \|f\|_{L^\infty}^{-2\gamma/3}\|f\|_{\dot{H}^1_{-(2+\gamma)/2}}^2\\
    &\lesssim \|f\|_{L^\infty}^{-\gamma/3} \|f\|_{H^2}^2 + \eps\|f\|_{\dot{H}^3_{(2+\gamma)/2}}^2 + \frac1\eps \|f\|_{L^\infty}^{-2\gamma/3}\left(\|f\|_{H^2}^2 + \|f\|_{L^2_\theta}^2\right)
\end{align}
which was the desired estimate \eqref{eq: Q2 beta_2=1 k=2 }. Here we have used Lemma \ref{l: hbound} repeatedly. To pass to the last line, we used Lemma \ref{l: app7}, with $\theta$ is taken large enough to apply. 
\\\par \textbf{Conclusion.} 
We want to put together all of the inequalities \eqref{eq: Q1estimatesecondslot}-\eqref{eq: Q2 beta_2=1} and plug into \eqref{eq: inductionderivative}. 
By Corollary \ref{cor: L^p_kestimate} and Proposition \ref{prop: momentsestimate}, there exists a sufficiently large $m>0$ such that 
\begin{equation}
    \norm{f}_{L^2_\theta}+ \norm{f}_{L^4_q}\lesssim \norm{f}_{L^\infty}+\norm{f}_{L^1_m} \lesssim K+\norm{f_\text{in}}_{L^1_m}.
\end{equation}

For some $\tilde{C}_{k-1}>0$, depending only on $\norm{f_\text{in}}_{L^1_m} $, $\norm{f_\text{in}}_{H^{k-1}},T,$ and $K$,  
\begin{equation}
    \frac{d}{dt} \int_{\R^3} |D^k f|^2 \dd v \le \frac{\tilde{C}_{k-1}}{\eps} \|f\|_{H^k}^2\left( 1+ \|f\|_{L^\infty}^2 +\norm{f}_{L^1_m}^2  +F_{k-1}(t)^2\right)
+ \eps F_{k}(t)^2 - \sum_{|\beta|=k} N_\gamma^f(\partial^\beta f)^2.
\end{equation}
Note that we used $\norm{f}_{L^\infty}^{-2\gamma/3} +\norm{f}_{L^\infty}^{-\gamma/3}+F_{k-1}(t) \lesssim 1+ \norm{f}_{L^\infty}^2 +F_{k-1}(t)^2.$
From Lemmas \ref{l: abound} and \ref{l: coercivity}, we find a constant $c_0>0$, depending on $\norm{f_\text{in}}_{L^1}$and $K$, such that 
$$\sum_{|\beta|=k} N_\gamma^f(\partial^\beta f)^2 \ge 2c_0 \int_{\R^3} \Jap{v}^{2+\gamma}|D^{k+1}f|^2 \dd v = 2c_0 F_{k}(t)^2.$$
Choose $\eps=c_0$ so that the $F_{k}(t)^2$ term may be absorbed into the left hand side. 
\begin{equation}
    \label{eq: penultimateinequality}
\frac{d}{dt}\|D^kf\|_{L^2(\R^3)}^2+ c_0 \int_{\R^3} \Jap{v}^{2+\gamma} |D^{k+1}f|^2 \dd v \le \frac{\tilde{C}_{k-1}}{c_0} \left(C_{k-1}^2
+\|D^kf\|_{L^2}^2 \right) \left(1+  F_{k-1}(t)^2 \right)
\end{equation}
where we have used $\|f\|_{H^k}^2= \|f\|_{H^{k-1}}^2 +\|D^kf\|_{L^2}^2 \le C_{k-1}^2+ \|D^kf\|_{L^2}^2  $ and where $P$ is polynomial in $K$. We apply Gr\"onwall's inequality to find that for $t\in [0,T]$,
\begin{equation}
\label{eq: finalinequality}
\|D^kf(t)\|_{L^2}^2 \le \left(\norm{D^k f_\text{in}}_{L^2}^2+ \frac{\tilde{C}_{k-1}}{c_0}C_{k-1}^2 \int_0^T (1+F_{k-1}(s)^2)\dd s \right)\exp \left( \frac{\tilde{C}_{k-1}}{c_0}\int_0^T (1+F_{k-1}(s)^2) \dd s \right) \le C_k.
\end{equation}
Indeed, $F_{k-1}^2 \in L^1([0,T])$ by the induction hypothesis.
 So \eqref{eq: inductionsobolev} is proven for $k$, and integrating \eqref{eq: penultimateinequality} in time we obtain \eqref{eq: inductionsobolevweight}. This completes the proof. 
\end{proof}
The growth estimates of this section will prove crucial to providing an $L^\infty$ continuation criterion for solutions to \eqref{eq: Kriegerpotential}, as we shall see in the next section.
\section{$L^\infty$ Local Existence Theory}
There does not seem to be a satisfying local existence theory for the Krieger-Strain equation \eqref{eq: Kriegerpotential}. The existing theorems in the literature concern initial data which is radially symmetric and monotone decreasing, see the citations in the introduction. In this section, we follow the ideas of Golding and Loher in \cite{golding2024localintime} to establish an $L^\infty$ short-time existence theory. To do so, we first need the intermediate step of constructing solutions for data with sufficient regularity. We do so with the following Proposition \ref{prop: localexistenceHK}, which establishes the short-time existence of solutions for initial data lying in weighted Sobolev spaces. 
\begin{prop}
    \label{prop: localexistenceHK}
    Let $\gamma \in [-3,-2)$ and $q>\frac72+\gamma$. For any non-negative initial condition $f_{\text{in}} \in H^{3}_q(\R^3)$, there exists $T>0$ depending only on $\|f_{\text{in}}\|_{H^3_q}$ and a unique non-negative solution
    $f\in L^\infty([0,T]; H^3_q(\R^3)) \cap W^{1,\infty}([0,T]; H_q^1(\R^3))$ to the Krieger-Strain equation
    $$\partial_tf = \mathcal{Q}_{KS}(f,f).$$
    If in addition $f_{\text{in}} \in H^k_{q'}(\R^3)$ for any $k\ge 3$ or $q'>q$, then for any $t \in [0,T]$, there holds
    $$\|f(t)\|_{H^k_{q'}} \le \|f_{\text{in}}\|_{H^k_{q'}}\exp(Ct(1+\|f\|_{L^\infty([0,t]; H^3_q}))), \ \ t\in [0,T].$$
    In particular, if $\Jap{v}^q f_{\text{in}} \in C^\infty(\R^3)$ for all $q>0$, then $\Jap{v}^q f \in C^\infty([0,T] \times \R^3)$ for all $q>0$. 
\end{prop}
The statement of this theorem is completely analogous to the corresponding Theorem $1.3$ of Snelson's work on the isotropic Boltzmann equation, \cite{snelson2023global}. The proof is quite standard. Following Snelson, we establish energy estimates for the linearized equation and then apply the method of continuity followed by an iteration procedure. Because it is not our main local existence result, the proof is postponed until Appendix $B$.
As a corollary, we obtain the following sort of $L^\infty$ continuation criterion. 
\begin{cor}
\label{cor: continuationcriterion1}
    Let $f_{\text{in}}$ be an initial data which is nonnegative, smooth, and belonging to $H^4(\R^3) \cap H^3_q(\R^3) \cap L^1_m(\R^3)$ for some $q>\frac72+\gamma$ and $m$ as in Proposition \ref{l: H^klemma}. Then there exists a classical solution $f:[0,T_*) \times \R^3 \to [0,\infty)$ to the equation \eqref{eq: Kriegerpotential}, where $[0,T_*)$ is the maximal interval of existence. If $T_*=+\infty$, $f$ is a global solution, and if $T_*<+\infty$, it is characterized by the blow-up criterion  
$$\lim_{t \nearrow T_*} \|f(t)\|_{L^\infty} = +\infty.$$
\end{cor}
\begin{proof}
    Proposition \ref{prop: localexistenceHK} yields a classical solution $f:[0,T_*) \to [0,\infty)$. If $T_*<\infty$, then it must be that $\|f(t)\|_{H^3_q} \to \infty$ as $t \to T_*$. Because we have assumed $f_{\text{in}} \in L^1_m(\R^3)$ for $m$ large enough, the growth estimates of Proposition \ref{l: H^klemma} imply that $\|f(t)\|_{L^\infty} \to \infty$ as $t \to T_*$. Indeed, if for all $t\in [0,T_*)$ it were the case that $\|f(t)\|_{L^\infty} \le K <\infty$ for some $K>0$, then $\|f(t)\|_{H^4}$ would stay bounded on $[0,T_*)$. Interpolating the $H^4$ estimates with the $L^1_m \cap L^\infty \hookrightarrow L^2_{m/2}$ estimates, we obtain a uniform estimate on $\|f(t)\|_{H^3_q}$ which does not blow up, reaching a contradiction. And for each time $t<T_*$, the Sobolev embedding implies $\|f(t)\|_{L^\infty}<\infty$, and therefore $T_*$ is the actual blow-up time for $\|f(t)\|_{L^\infty}$. 
\end{proof}
Our first step towards a satisfying $L^\infty$ existence theorem is to prove a small-time growth estimate for $\|f(t)\|_{L^\infty}$ that relies only on $\|f_{\text{in}}\|_{L^\infty}$. 

\begin{lem} 
\label{lem: L^inftysmalltimegrowth}
    Let $f\in C^\infty((0,T)\times \R^3)$ be a nonnegative solution of \eqref{eq: Kriegerpotential}
    with initial data $f_{\text{in}}\in L^\infty(\R^3).$ Suppose further that the maximum of $f(t,\cdot)$ is attained for each $t>0$. 
    Then for almost every $t \in [0,T]$,
    $$\partial_t \|f(t)\|_{L^\infty} \le C\|f(t)\|_{L^\infty}^{1-\gamma/3}$$
    with constant $C$ depending only on $\gamma$ and $\|f_{\text{in}}\|_{L^1}$. As a consequence, for all $0 \le t \le \displaystyle \frac{3(1-2^{-\gamma/3})\|f_{\text{in}}\|_{L^\infty}^{\gamma/3}}{2^{-\gamma/3}\gamma C}$, it holds that
    $$\|f(t)\|_{L^\infty} \le 2 \|f_{\text{in}}\|_{L^\infty}.$$
\end{lem}
\begin{proof}
    See \cite{968a84f7-1561-3553-9573-cf52054db1c4} for a rigorous argument, which applies directly to our context with no modification necessary. The spirit of the argument is that at a maximum $v\in \R^3$ of $f(t,\cdot)$, it holds that $\Delta f(t,v) \le 0$ and therefore
    $$\partial_t\|f(t)\|_{L^\infty} = \partial_t f(t,v)  = a[f]\Delta f - (2+\gamma)h[f]f \le -(2+\gamma)h[f]f \lesssim \|f_{\text{in}}\|_{L^1}^{1+\gamma/3}\|f(t)\|_{L^\infty}^{1-\gamma/3}$$
    by Lemma \ref{l: hbound}. 
\end{proof}
We will use this estimate for smooth solutions with $H^3_q(\R^3) \cap L^2_m(\R^3)$ regularity and decay. By the Sobolev embedding type result of Lemma \ref{l: app5}, such solutions actually belong to $L^\infty_q(\R^3)$ for each time $t$. This guarantees decay at infinity, which is enough to say that the maximum is attained in the above proof. 

Now we arrive at the following $L^\infty$ type existence theorem. It is completely analogous to the existence theorem of \cite{golding2024localintime}. For this reason, we only sketch the proof, following their outline. We remark that we make no assumption on the entropy dissipation, since in our paper the entropy dissipation is effectively replaced by the monotonicity of the Fisher information. 

Now we establish the local existence result of Theorem \ref{thm: Linftylocaltheory}. 

\begin{proof}[Proof of Theorem \ref{thm: Linftylocaltheory}]
    
    Given $f_{\text{in}}$ satisfying
    \begin{equation}
        \norm{f_\text{in}}_{L^\infty(\R^3)} + \int_{\R^3} f_\text{in} \Jap{v}^m \dd v \le M \label{eq: fincondition}
    \end{equation}
    for some $M\ge 0$, let us approximate $f_{\text{in}}$ with a mollification sequence $f^\eps_{\text{in}} = f_{\text{in}} \ast \rho_\eps$, with $\rho_\eps$ being a standard mollifier. As $f_{\text{in}} \ge 0$, it holds that $f_{\text{in}}^\eps \ge 0$ and $\int_{\R^3} f_{\text{in}}^\eps \dd v = \int_{\R^3} f_{\text{in}} \dd v$. Additionally, $f_\text{in}^\eps \rightharpoonup^* f_{\text{in}}$ in $L^\infty(\R^3)$, so we obtain by Jensen's inequality and the weak$-\ast$ lower semicontinuity of the $L^\infty$ norm the convergence of the norms $\|f_{\text{in}}^\eps\|_{L^\infty} \to \|f_{\text{in}}\|_{L^\infty}$.

    Further, $f_{\text{in}}^\eps \to f_{\text{in}}$ in $L^1_m \cap L^p(\R^3)$ for every $1\le p <\infty$. In light of these convergences, it holds for all $\eps>0$ small enough that
    \begin{equation}
       \norm{f_\text{in}^\eps}_{L^\infty(\R^3)} + \int_{\R^3} f_\text{in}^\eps \Jap{v}^m \dd v\le 2M.
    \end{equation}
    Passing to a further subsequence if necessary, we have pointwise almost everywhere convergence $f_{\text{in}}^\eps \to f_{\text{in}}$ as well. Finally, $f_{\text{in}}^\eps \in C^\infty(\R^3) \cap H^k_{m/2}(\R^3)$ for any $k\in \mathbb{N}$, since $L^1_m \cap L^\infty(\R^3) \hookrightarrow L^2_{m/2}(\R^3).$ Since we are taking $m$ large enough, we may apply Proposition \ref{prop: localexistenceHK} to obtain the existence of smooth classical solutions $f_\eps:[0,T_\eps) \to [0,\infty)$ to \eqref{eq: Kriegerpotential} with initial data $f^\eps_{\text{in}}$. Here $T_\eps>0$ is the maximal time of existence for $f_\eps$, which by Corollary \ref{cor: continuationcriterion1}, is such that either $T_\eps= +\infty$ or $T_\eps <+\infty$ and is characterized by
    $$\lim_{t \nearrow T_\eps} \|f_\eps(t)\|_{L^\infty} =\infty.$$

In order to construct from the family $\{f_\eps\}$ a solution to \eqref{eq: Kriegerpotential} with initial data $f_{\text{in}}$, we need first to obtain a lower bound on the maximal times of existence $T_\eps$. The estimate of Lemma \ref{lem: L^inftysmalltimegrowth} ensures that 
$$\|f_\eps(t)\|_{L^\infty} \le 2\|f^\eps_{\text{in}}\|_{L^\infty} \le 4M$$
for all $0 \le t \le \frac{3(1-2^{-\gamma/3})(2M)^{\gamma/3}}{2^{-\gamma/3}\gamma C} \coloneqq T_0>0$. Therefore, we obtain $T_\eps \ge T_0$ for all $\eps>0$. We now proceed to construct a solution on $[0,T_0] \times \R^3$ in a sequence of steps.
\\\par
\textbf{Step One: Uniform Estimates.} We claim that the family $f_\eps: [0,T_0] \to \R^+$ constructed above satisfies 
    $$\|f_\eps\|_{C^k(t_0,T_0)\times B_R} \le C(M,R,t,k)$$
    uniformly in $\eps$, for each $k \in \mathbb{N}$, $t_0>0$, and $R>0$. We prove this using the parabolic regularity theory and bootstrapping. The first observation is that the family $\{f_\eps\}$ is uniformly bounded in $L^\infty([0,T_0]; L^3(\R^3))$.
    This is obtained by interpolating the conservation of mass $\|f_\eps(t)\|_{L^1}= \|f_{\text{in}}\|_{L^1}$ for any $\eps>0, t\ge 0$, with the uniform $L^\infty$ estimates on $[0,T_0]$ thanks to Lemma \ref{lem: L^inftysmalltimegrowth} and the choice of $T_0$. In doing so we obtain that $a[f_\eps]$ satisfies the ellipticity estimates of Lemma \ref{l: abound}, uniform in $\eps$. In particular, $a[f_\eps](t,v) \ge c(M,R)$ on $[0,T_0] \times B_R$. By Lemma \ref{l: hbound}, $ h[f_\eps]f_\eps \in L^\infty([0,T_0]\times \R^3)$ with the uniform bound $$\|h[f_\eps]f_\eps\|_{L^\infty([0,T_0]\times \R^3)}\lesssim \|f_{\text{in}}\|_{L^1} ^{1+\gamma/3} \|f_{\text{in}}\|_{L^\infty}^{1-\gamma/3}.$$
    Therefore we find that on $[0,T_0] \times B_R$, the solutions $f_\eps$ satisfy a uniformly parabolic non-divergence form equation with bounded right-hand side:
    $$\partial_t f_\eps -a[f_\eps]\Delta f_\eps = -(2+\gamma)h[f_\eps]f_\eps \in L^\infty([0,T_0]\times B_R).$$
    Applying the Krylov-Safonov theory, we obtain $f_\eps \in C^\alpha((t_0,T_0) \times B_R)$, with an estimate independent of $\eps$. As in \cite{golding2024localintime}, we then iteratively apply the Schauder theory in a standard way; each time showing that the diffusion coefficients $a[f_\eps]$ and the reaction terms $h[f_\eps] f_\eps$ are more regular (see Lemma $5.1$ of \cite{golding2024localintime} to see how pass regularity estimates from a function $g$ to $a[g]$ and $h[g]$). 
\\\par
\textbf{Step Two: Compactness.} Applying standard compactness results, we may pass to a limit $f_\eps \to f:[0,T_0] \times \R^3\to [0,\infty)$, with the convergence happening in each of the following senses:
    \begin{itemize}
        \item Strongly in $C^k((t_0,T_0)\times B_R)$ for each $k \in \mathbb{N}, R>0,$ and $0<t_0<T_0.$
        \item Pointwise a.e., in $[0,T_0] \times \R^3,$
        \item Weak-$\ast$ in $L^\infty([0,T_0] \times \R^3)$
        \item Strongly in $L^1([0,T_0]\times \R^3).$
          \end{itemize}
        The justifications are standard, and the same as in \cite{golding2024localintime}. The first bullet point implies that the limit $f$ is nonnegative, $f\ge 0$, and $C^\infty$ regular on $(0,T_0)\times \R^3$. The latter three points are used to justify that the limit $f$ is a solution. Notice that the strong $L^1_{t,v}$ convergence is a corollary of the pointwise convergence and the uniform integrability guaranteed by the $L^1_m$ bound. We remark that the entropy bound of \cite{golding2024localintime} is not necessary for our purposes, which is why we do not assume it.   
\\\par
\textbf{Step Three: Passing to a Limit.} We use the viscosity formulation. Because the limit $f$ is smooth for positive time, we just need to show that it is a viscosity solution to \eqref{eq: Kriegerpotential}. Because of the $C^k_{loc}((0,T_0) \times \R^3)$ convergence, it suffices to prove that the nonlocal terms $a[f_\eps] \to a[f]$ and $h[f_\eps] \to h[f]$ converge locally uniformly. Let us focus on the convergence of the more singular terms, $h[f_\eps] \to h[f]$. We aim to show that $h[f_\eps] \to h[f]$ uniformly on $(\tau_0, \tau_1) \times B_{R}$ for arbitrary $0<\tau_0<\tau_1<T_0$ and $R>0$. This is a simple corollary of the fact that $f_\eps \to f$ locally uniformly. Indeed, fix $t \in (\tau_0,\tau_1)$ and $|v|<R$. For $\gamma=-3$, there is nothing to do. For $\gamma \in (-3,-2)$, we find
\begin{align}|h[f_\eps](t,v) - h[f](t,v)|& \le \int_{B_{R}(v)} |v-w|^{\gamma}|f_\eps(t,v) - f(t,v)| \dd v + \int_{\R^3 \setminus B_{R}(v)}|v-w|^\gamma |f_\eps(t,v)-f(t,v)| \dd v\\
&\le C(\gamma)\|f_\eps(t)-f(t)\|_{L^\infty(B_{R}(v))} R^{3+\gamma} + R^\gamma \|f_\eps(t)-f(t)\|_{L^1} \\
&\le   C(\gamma)R^{3+\gamma}\|f_\eps(t)-f(t)\|_{L^\infty(B_{2R})}+2\|f_{\text{in}}\|_{L^1}R^\gamma.  \end{align}
First pick $R>0$ large so that the latter term is small. Then for this $R$, send $\eps \to 0$ and use the local uniform convergence of $f$ to make the former term small. We have shown that $h[f_\eps] \to h[f]$ uniformly on $(\tau_0, \tau_1) \times B_{R}$ for $R$ large, which obviously implies the local uniform convergence for all $R$. The local uniform convergence of the coefficients ensures that the limit $f$ is a viscosity solution to the Krieger-Strain equation \eqref{eq: Kriegerpotential}. Then because $f$ instantaneously becomes smooth, it is a classical solution on $(0,T_0)\times \R^3$.
\\\par
\textbf{Step Four: Behavior at the Initial Time.} We argue that $f(t) \rightharpoonup^* f_{\text{in}}$ in $L^\infty(\R^3)$ as $t\to 0$ by compactness. First we claim that $\partial_t f_\eps$ is uniformly bounded in $\eps$ in the space $L^\infty([0,T_0]; W^{-2,\infty}(\R^3))$. Using the equation, it suffices to check that $a[f_\eps(t)] \Delta f_\eps(t)$ and $h[f_\eps(t)] f_\eps(t)$ are bounded in $W^{-2,\infty}(\R^3)$, uniformly in $t$ and $\eps$. The latter term $h[f_\eps(t)]f_\eps[t]$ is in fact bounded in $L^\infty(\R^3) \subset W^{-2,\infty}(\R^3)$, uniformly in $t \in [0,T_0]$ and $\eps>0$. Indeed, $$\|f_\eps(t) h[f_\eps(t)]\|_{L^\infty} \lesssim  \|f(t)\|_{L^1}^{1+\gamma/3}\|f_\eps(t)\|_{L^\infty}^{1-\gamma/3} \lesssim \|f_{\text{in}}\|_{L^1}^{1+\gamma/3}\|f_{\text{in}}\|_{L^\infty}^{1-\gamma/3}$$
thanks to the conservation of $L^1$ norm, Lemma \ref{l: hbound}, and the choice of $T_0$. To handle $a[f_\eps] \Delta f_\eps$, fix $\varphi \in C_c^\infty(\R^3)$ and observe
\begin{align}\langle a[f_\eps] \Delta f_\eps, \varphi\rangle_{W^{-2,\infty},W^{2,1}}&=\int_{\R^3} f_\eps \Delta(a[f_\eps]\varphi) \dd v\\
&= \int_{\R^3} f_\eps\left((2+\gamma)h[f_\eps] \varphi + 2\nabla a[f_\eps] \cdot \nabla \varphi + a[f_\eps]\Delta\varphi\right)\\
&\lesssim C(\|f_{\text{in}}\|_{L^\infty}) \|\varphi\|_{W^{2,1}}
\end{align}
as $\|f \ast|\cdot|^{\gamma}\|_{L^\infty} + \|f \ast|\cdot|^{1+\gamma}\|_{L^\infty} +\|f \ast|\cdot|^{2+\gamma}\|_{L^\infty}$ may be controlled by $\|f\|_{L^1}$ and $\|f\|_{L^\infty}$, as per Lemmas \ref{l: abound} and \ref{l: hbound}. 
This verifies our claim that $\partial_t f_\eps$ is bounded uniformly in $\eps>0$ in the space $L^\infty([0,T_0]; W^{-2,\infty}(\R^3))$. 
As $f_\eps$ is uniformly bounded in $\eps$ in $L^\infty([0,T_0]; L^\infty(\R^3))$ by our choice of $T_0$, we may apply the Aubin-Lions Lemma to deduce $f_\eps$ is compact in $C([0,T_0]; W^{-1,\infty}(\R^3))$. Necessarily, the full sequence must converge to $f$. This combined with $f^\eps_{\text{in}} \rightharpoonup^* f_{\text{in}}$ allows us to deduce that $f$ is weak$-\ast$ continuous in time taking values in $W^{-1,\infty}(\R^3)$, with limit $f_{\text{in}}$ as $t\to 0$. By the Banach-Alaoglu theorem, the same limit must be attained in the weak$-\ast$ $L^\infty$ sense. 
\\\par
\textbf{Step Five: Monotonicity of the Fisher Information.}
Suppose further that $f_{\text{in}}$ has finite Fisher information, i.e., $i(f_{\text{in}}) <\infty$. We aim to show that for the solution $f$ constructed above, the Fisher information $i(f(\cdot))$ decreases as a function of time. To do so, we use the same approximation $f_\eps$. First of all, because the Fisher information is convex, it holds by Jensen's inequality that
$$i(f_{\text{in}}^\eps) \le i(f_{\text{in}})<\infty.$$
Furthermore, $\sqrt{f^\eps_{\text{in}}}$ is bounded in $H^1(\R^3)$ with $\|f^\eps_{\text{in}}\|_{L^2}^2 = \int_{\R^3} f_{\text{in}}$ for any $\eps>0$. Weak lower semicontinuity of the $H^1$ norm implies 
$$i(f_{\text{in}}) = 4\|\sqrt{f_{\text{in}}}\|_{\dot{H^1}}^2 \le \liminf_{\eps \to 0} 4\|\sqrt{f^\eps_{\text{in}}}\|_{\dot{H^1}}^2 = \liminf_{\eps \to 0} i(f_{\text{in}}^\eps),$$
so in fact we have $i(f_{\text{in}}^\eps) \to i(f_{\text{in}})$ as $\eps \to 0$.

In Theorem $1.3$ of \cite{golding2024global}, they extend the monotonicity of the Fisher information for the Landau-Coulomb equation to initial data belonging to $C^\infty \cap L^1_m(\R^3)$. Since we have the exact same tools, an identical proof applies in our case. In particular, for the solutions $f_\eps$ corresponding to the initial data $f_\text{in}^\eps = f_{\text{in}}\ast \rho_\eps \in C^\infty \cap L^1_m(\R^3)$, it holds that $i(f_\eps(t))$ decreases as a function of time. In particular, for each $t>0$, 

$$\|\sqrt{f_\eps(t)}\|_{H^1}^2 = \|f_{\text{in}}\|_{L^1} +4 i(f_\eps(t)) \le \|f_{\text{in}}\|_{L^1}  +4i(f_{\text{in}}). $$
By the same reflexivity and weak lower semicontinuity argument, we obtain
$$i(f(t)) \le \liminf_{\eps \to 0} i(f_\eps(t)) \le \liminf_{\eps \to 0} i(f_\text{in}^\eps) \le i(f_{\text{in}})<\infty.$$
Now, fix $s \in (0,T_0)$, and let $g_\eps$ solve the equation \eqref{eq: Kriegerpotential} with initial data $f_\eps(s)\in H^3_q\cap C^\infty(\R^3).$ By the uniqueness guaranteed by Proposition \ref{prop: localexistenceHK}, the solutions $g_\eps(t) = f_\eps(t+s)$, and in particular $g_\eps(t) \to f(t+s)$ in $C^k_{\text{loc}}(\R^3)$ for any $k \in \mathbb{N}$. Applying the method of Golding, Gualdani, and Loher from Theorem $1.3$ of \cite{golding2024global}, we have that $i(g_\eps(\cdot))$ decreases as a function of time. Applying the reasoning of the previous paragraphs, we find for any $s>0$ and $0<t<T_0-s$, there holds 
$$i(f(t+s)) \le i(f(s)).$$

This is the desired monotonicity of the Fisher information. Let us quickly remark that the Jensen's inequality based reasoning of the beginning was enough to show that $i(f(\cdot))$ is bounded in time, which would suffice in proving global existence. To actually get monotonicity, we no longer had that $f_\eps(t)$ was a mollification of $f(t)$. This is where it became convenient to source from \cite{golding2024global}.
\end{proof}

\section{Concluding the Global Existence Theorem}
In this section we briefly put together the proof of the global existence theorem. It is just a matter of patching together the results of each section.

\begin{proof}[Proof of Theorem \ref{thm: globalexistencethm}.]

We begin with  $0\le f_{\text{in}} \in L^\infty(\R^3) \cap L^1_m(\R^3)$ with sufficiently large $m$, and assume further that $f_{\text{in}}$ has finite Fisher information, i.e., $i(f_{\text{in}}) <+\infty$. The short-time existence result of Theorem \ref{thm: Linftylocaltheory} furnishes a classical solution $f:[0, T_*) \times \R^3\to [0,\infty)$ of the Krieger-Strain equation \eqref{eq: Krieger-Strain,divform2}, for some maximal interval $[0,T_*)$. It must be shown that the assumption $i(f_{\text{in}})<+\infty$ implies $T_*=+\infty$. Recall from the local theory, 
$$T_* <\infty \text{ if and only if } \lim_{t\nearrow T_*} \|f(t)\|_{L^\infty} = \infty.$$
We begin by noting that by the arguments of \cite{henderson2019local, henderson2020local},  and the local theory, $f$ is smooth and strictly positive on $(0,T_*) \times \R^3$. Next, Theorem \ref{thm: Linftylocaltheory} ensures that $i(f(\cdot))$ is monotone decreasing in time, from which we apply Corollary \ref{cor: L3estimate} to obtain that $f \in L^\infty([0,T_*); L^3(\R^3))$. This control combined the moment estimates of Proposition \ref{prop: momentsestimate} allow us to apply Theorem \ref{thm: kriegerstraininfinitybound} to deduce that $\|f(t)\|_{L^\infty}$ stays bounded as $t$ increases, and as a consequence the finite time blow-up criterion is never met. 
\end{proof}

\appendix
\section{Technical Inequalities}

We make extensive use of the following elementary but important lemmas. The first two lemmas are taken verbatim from \cite{snelson2023global}. For the body of the paper, we always assume $d=3$.

\begin{lem} \label{l: app1}
    Let $f: \R^d \to \R$ and $1\le p \le \infty$. Let $p'$ be the H\"older conjugate of $p$. If $-\frac{d}{p'}<\mu<0$ and $m>\frac{d}{p'}+\mu$, then
    \begin{equation}
        \norm{f\ast |\cdot|^\mu}_{L^\infty(\R^d)} \le C\|f\|_{L^p_m(\R^d)}
    \end{equation}
    for some $C=C(d,\mu,p,m).$
    In particular,
    \begin{itemize}
        \item If $f \in L^\infty_m(\R^d)$ for some $m>d+\mu$, then
        $$\|f \ast |\cdot|^\mu\|_{L^\infty(\R^d)} \le C\|f\|_{L^\infty_m(\R^d)},$$
        for some $C = C(d,\mu, m)>0$. 
        \item If $f \in L^2_m(\R^d)$ for some $m>\frac{d}{2} + \mu>0$, then 
        $$\|f\ast |\cdot|^\mu\|_{L^\infty(\R^d)} \le C\|f\|_{L^2_m(\R^d)}$$
        for some $C=C(d,\mu,m)>0$. 
    \end{itemize}
\end{lem}

\begin{lem} \label{l: app3}
For $-d<\mu\le -d/2, \theta>\mu+d$, and $g \in L^2_\theta(\R^d)$, one has
$$\|g\ast |\cdot|^\mu\|_{L^2(\R^d)} \le C\|g\|_{L^2_\theta(\R^d)}.$$
\end{lem}
\begin{lem} \label{l: app4}
Let $1 \le p \le \infty$. Suppose that $-\frac{d}{p'}<\mu<0 $ and $m >  \frac{d}{p'}$. Then there exists a constant $C = C(\mu,d, p)$ such that 
    \[
    \norm{f\ast |\cdot|^\mu }_{L^\infty_{-p'\mu}(\R^d)} \le C  \norm{f}_{L^p_m(\R^d)}.
    \]
\end{lem}
\begin{proof}
    It is equivalent to 
    \begin{equation}
        \norm{f\ast |\cdot|^\mu }_{L^\infty(\R^d)} \le C \Jap{v}^{p'\mu} \norm{f}_{L^p_m(\R^d)}.
    \end{equation}
    From Lemma \ref{l: app1}, we have that 
    \begin{equation} \label{in: app4-1}
         \norm{f\ast |\cdot|^\mu }_{L^\infty(\R^d)} \le C  \norm{f}_{L^p_m(\R^d)}.
    \end{equation}
    On the other hand, write
    $$(f\ast |\cdot|^\mu) (v) \le  \|f\|_{L^p_m(\R^d)} \left(\int_{\R^d} |w|^{p'\mu} \Jap{v-w}^{-p'm}\dd w \right)^{1/p'}.$$
    To estimate the latter term, split the integral into 
    $$\int_{\R^d} |w|^{p'\mu} \Jap{v-w}^{-p'm} \dd w= \int_{|w|<v/2} |w|^{p'\mu} \Jap{v-w}^{-p'm} \dd w + \int_{|w|\ge |v|/2} |w|^{p'\mu} \Jap{v-w}^{-p'm} \dd w \coloneqq I_1 + I_2.$$
    Where $|w| <|v|/2$, we have $|v-w|>|v|/2$ and so $\Jap{v-w}^{-p'm} \lesssim \Jap{v}^{-p'm}$, so 
    $$I_1 \lesssim \Jap{v}^{-p'm} |v|^{p'\mu + d}\lesssim \Jap{v}^{-p' m + p' \mu + d} \le \Jap{v}^{p' \mu}.$$
    thanks to $\frac{d}{p'}+\mu  >0$.
    For $I_2$,
    \begin{equation}
        I_2 \le \left(\frac{|v|}{2}\right)^{p' \mu} \int_{|w|\ge |v|/2} \Jap{v-w}^{-p'm} \dd w \lesssim |v|^{p' \mu} \int_{\R^d} \Jap{w}^{-p'm} \dd w  
    \end{equation}
    Since $p'm>d$, we get
    \begin{equation} \label{in: app4-2}
        \int_{\R^d} \Jap{w}^{-p'm} \dd w <+\infty.
    \end{equation}
    Hence, from \eqref{in: app4-1} and \eqref{in: app4-2},
    \begin{equation}
        (f\ast |\cdot|^\mu) (v) \le C \min(1,|v|^{p'\mu}) \|f\|_{L^p_m(\R^d)} \le C \Jap{v}^{p'\mu} \|f\|_{L^p_m(\R^d)}.
    \end{equation}
\end{proof}
The following lemma ensures that we may interpolate with negative weights, which is a need that naturally arises in proving the propagation of $H^k$ norms.
\begin{lem}
\label{l: app7}
For $q,\theta \ge 0$, and $\delta>0$
    $$\|\nabla f\|_{L^2_q} \lesssim \frac1\delta \|f\|_{L^2_{2q+\theta}} +\delta \|D^2f\|_{L^2_{-\theta}}.$$

    As a corollary, for all $k \in \mathbb{N}$ there exists a constant $C= C(k,d)$ and an exponent $m$ depending on $k,q$ and $\theta$ such that
    $$\|f\|_{\dot{H}^k_q} \le C(\|f\|_{L^2_m} + \|f\|_{\dot{H}^{k+1}_{-\theta}}).$$
\end{lem}

\begin{proof}
Integrating by parts, 
    \begin{align}\|\nabla f\|_{L^2_q}^2 &=\int \Jap{v}^{2q} |\nabla f|^2 = -2q\int f\Jap{v}^{2q-2} v\cdot \nabla f - \int \Jap{v}^{2q+\theta}f \Jap{v}^{-\theta} \Delta f\\
    &\lesssim_q \eps\|\nabla f\|_{L^2_q}^2 + \frac1\eps \|f\|_{L^2_{q-1}}^2 + \frac1\delta\|f\|_{L^2_{2q+\theta}}^2 + \delta \|D^2f\|_{L^2_{-\theta}}^2.\end{align}

    Choosing $\eps>0$ small enough we find
    $$\|\nabla f\|_{L^2_{q}} \lesssim_q \frac1\delta \|f\|_{L^2_{2q+\theta}} +\delta \|D^2f\|_{L^2_{-\theta}},$$
    for $\delta>0$ a parameter at our disposal. 

The general case follows from the following argument
$$\|f\|_{\dot{H}^k_q} \lesssim \|f\|_{\dot{H}^{k-1}_{2q+\theta}}+ \|f\|_{\dot{H}^{k+1}_{-\theta}} \lesssim \delta \|f\|_{\dot{H}^k_q} + \frac1\delta\|f\|_{\dot{H}^{k-2}_m} + \|f\|_{\dot{H}^{k+1}_{-\theta}}.$$

Choose $\delta>0$ small to absorb the $\|f\|_{\dot{H}^k_q}$ norm into the lefthand side and continue iterating down to $k=0$. 
\end{proof}
Since at various points we have needed to control weighted $L^\infty_m$ norms with Sobolev norms, we also need the following lemma.
\begin{lem}
\label{l: app5}
For any $m \in \R$ the embedding $H^2_m(\R^3) \subset L^\infty_m(\R^3)$ holds with the following inequality:
$$\|f\|_{L^\infty_m} \lesssim_m   \|f\|_{L^2_m}+\|f\|_{\dot{H}^2_m}  \le \|f\|_{H^2_m}.$$
\end{lem}
The only point of this proof is to show that we may control from above $\|\Jap{v}^m f\|_{H^2}$, which is the definition of the $H^2_m$ norm taken in \cite{snelson2023global}, by our choice of the $H^2_m$ norm, which is $\sum_{j=0}^2 \| D^jf\|_{L^2_m}$. Of course, this will hold for any $H^k_m$ and any dimension, but we prove it for $k=2, \ \text{dim}=3$ out of convenience. 
\begin{proof}
The standard Sobolev embedding $H^2(\R^3)\subset L^\infty(\R^3)$ yields
$$\|f\|_{L^\infty_m} \lesssim \|\Jap{v}^mf\|_{H^2} \lesssim \|f\|_{L^2_m} + \|D^2(\Jap{v}^m f)\|_{L^2}$$
after interpolating out the first order term. 

Seeing as 
\begin{align}D^2(\Jap{v}^mf) &= \Jap{v}^m D^2f + 2\nabla(\Jap{v}^m) \otimes \nabla f + f D^2(\Jap{v}^m)\\
&= \Jap{v}^m D^2f + 2m \Jap{v}^{m-2}v\otimes \nabla f + f(m\Jap{v}^{m-2}\text{Id} + m(m-2)\Jap{v}^{m-4}v\otimes v),
\end{align}
we may control
$$\|D^2(\Jap{v}^mf)\|_{L^2} \lesssim_m \|f\|_{\dot{H}^2_m} + \|\nabla f\|_{L^2_{m-1}} + \|f\|_{L^2_{m-2}}.$$

To control $\|\nabla f\|_{L^2_{m-1}}$, we use the usual integration by parts and Cauchy's inequality argument:
\begin{align}
    \|\nabla f\|_{L^2_{m-1}} &= \int_{\R^3} |\nabla f|^2 \Jap{v}^{2(m-1)} \dd v = -\int_{\R^3} 2(m-1)\Jap{v}^{2(m-2)}f v \cdot \nabla f \dd v -\int f \Delta f \Jap{v}^{2(m-1)} \dd v\\
    &\lesssim \eps \|\nabla f \|_{L^2_{m-1}}^2 + \frac1\eps \|f\|_{L^2_{m-2}}^2 + \|f\|_{L^2_{m-1}}^2 + \|D^2f\|_{L^2_{m-1}}^2.
\end{align}
Taking $\eps>0$ small enough and cheaply increasing weights, we find $\|\nabla f\|_{L^2_{m-1}} \lesssim \|f\|_{L^2_m} + \|D^2f\|_{L^2_m} \le \|f\|_{H^2_m}.$
Cheaply increasing weights again, we  obtain
$$\|f\|_{L^\infty_m} \lesssim_m   \|f\|_{L^2_m}+\|f\|_{\dot{H}^2_m}  \le \|f\|_{H^2_m}.$$
as desired. 
\end{proof}
\section{Proof of Proposition \ref{prop: localexistenceHK}: $H^3_q$ Local Existence Theory}
We establish the local existence theory for $H^3_q$ initial data.
To do so, we follow the work of Snelson on the isotropic Boltzmann equation in \cite{snelson2023global}. The entire argument therein transfers almost directly to the isotropic Landau case. In fact, it is even simpler. Because of this, we merely provide a sketch of the proof, pointing out the differences when they arise. The method is standard: we linearly interpolate between \eqref{eq: Kriegerpotential} and the heat equation and prove the propagation of weighted Sobolev norms for solutions to such equations, and then apply the method of continuity and an iteration argument. 
\begin{lem} \label{lem: local}
    Let $k \ge 3$ and $q >0$ and  $m > \frac72+\gamma$ and $T>0$ be fixed, and let $f \in L^\infty([0,T]; H^k_m(\R^3))$ and $R \in L^2([0,T]; L^2_q(\R^3))$ be fixed. Assume further that $f\ge 0$. For $\sigma \in [0,1]$, let $g_\sigma$ be the solution of
    \begin{equation}
        \label{eq: heatinterpeq} 
        \partial_t g_\sigma= \sigma \mathcal{Q}_{KS}(f,g_\sigma) + (1-\sigma)\Delta g_\sigma + R,
    \end{equation}
    on $[0,T] \times \R^3$. Then
    \begin{equation} \label{eq: shorttimeestimate}
    \|g_\sigma(t)\|_{H^k_q(\R^3)} \le \left(\|g_\sigma(0)\|_{H^k_q(\R^3)} + C\int_0^t \|R(s)\|_{L^2_q(\R^3)}^2 \ ds\right)\exp\left(C\int_0^t 1+\|f(s)\|_{H^k_m(\R^3)} \ ds\right)
    \end{equation}
    for $C = C(\gamma, k, q,m)$ absolute and independent of $\sigma$. 
\end{lem}
The proof is very similar to both the proof of global existence and the proof of Lemma $6.2$ in \cite{snelson2023global}. 
\begin{proof}
    For aesthetic purposes, we denote $g_\sigma$ simply by $g$. Differentiate the equation by $\partial^\beta$ for a fixed multi-index $\beta$ of order $|\beta|\le k$. Then multiply by $\Jap{v}^{2q}\partial^\beta g$. We get
    \begin{align}
    \label{eq: firststepshorttime}
        \frac12 \partial_t \int_{\R^3} \Jap{v}^{2q} (\bdary^\beta g)^2 \dd v &= \sigma \int_{\R^3} \Jap{v}^{2q} \partial^\beta \mathcal{Q}_{KS}(f,g) \dd v + (1-\sigma)\int_{\R^3} \Jap{v}^{2q} \bdary^\beta g \Delta \bdary^\beta g \dd v\\ 
        &+ \int_{\R^3} \Jap{v}^{2q} R \bdary^\beta g \dd v.
    \end{align}
The latter two terms are straightforward. As in \cite{snelson2023global}, 
\begin{align}
\label{eq: snelsonestimate1}
    \int_{\R^3} \Jap{v}^{2q} R \bdary^\beta g \dd v\le \|R\|_{L^2_q}\|g\|_{H^k_q} \le \frac12 \|R\|_{L^2_q}^2 + \frac12 \|g\|_{H^k_q}^2
\end{align}
and 
\begin{align}
   \label{eq: snelsonestimate2}\int_{\R^3} \Jap{v}^{2q} \bdary^\beta g \Delta \bdary^\beta g \dd v \le C_{q} \|g\|_{H^k_q}^2
\end{align}
where the proof is the exact same. The real work is in handling the $\partial^\beta \mathcal{Q}_{KS}(f,g)$ term. Recall the definition of $\mathcal{Q}_1,\mathcal{Q}_2$ from \eqref{eq: KriegerQ1Q2}. By the product rule, 
$$\partial^\beta \mathcal{Q}_1(f,g) = \sum_{\beta_1 + \beta_2 = \beta} \partial^{\beta_1}a[f] \partial^{\beta_2}\Delta g$$
and
$$\partial^\beta \mathcal{Q}_2(f,g) = -(2+\gamma)\sum_{\beta_1 + \beta_2 = \beta} \partial^{\beta_1}h[f] \partial^{\beta_2}g.$$
We handle the $\mathcal{Q}_1$ and $\mathcal{Q}_2$ terms separately, using the above identities to estimate the first term of \eqref{eq: firststepshorttime}. As in Section $6$, we have to take care estimating depending on where the derivatives fall. 
\\\par Beginning with the $\mathcal{Q}_1$ terms, we will prove in all cases there holds
\begin{equation} \label{eq: Q1estimateshorttime}
\int_{\R^3}  \Jap{v}^{2q} \bdary^\beta g \mathcal{Q}_1(\bdary^{\beta_1}f, \bdary^{\beta_2}g) \dd v = \int_{\R^3} \Jap{v}^{2q} \bdary^\beta g \bdary^{\beta_1}a[f] \Delta \bdary^{\beta_2}g \dd v \lesssim \|g\|_{H^k_q}^2\|f\|_{H^k_m}.
\end{equation}

\textbf{Case One: Estimating the $\mathcal{Q}_1$ terms when $|\beta_2| \le k-2$.} In this case, we find immediately that for $m>\frac72 + \gamma$, there holds
\begin{align}
    \int_{\R^3}  \Jap{v}^{2q} \bdary^\beta g \bdary^{\beta_1}a[f] \Delta \bdary^{\beta_2}g \dd v  & \lesssim \|g\|_{H^{|\beta|}_q} \|g\|_{H^{2+|\beta_2|}_q} \|\bdary^{\beta_1}a[f]\|_{L^\infty}\\
    &\lesssim \|g\|_{H^k_q}^2 \|f\|_{H^{|\beta_1|}_m} \le\|g\|_{H^k_q}^2 \|f\|_{H^k_m},
\end{align}
applying Lemma \ref{l: derivative estimates a and h} to pass from the first line to the second. 

\textbf{Case Two: Estimating the $\mathcal{Q}_1$ terms when $|\beta_2| = k$.} Integrating by parts, we find

\begin{align}
\int_{\R^3}  \Jap{v}^{2q} \bdary^\beta g a[f] \Delta \bdary^\beta g \dd v &=-\int_{\R^3}  |\nabla \bdary^\beta g|^2 \Jap{v}^{2q}a[f] \dd v -\int_{\R^3} \bdary^\beta g \nabla \bdary^\beta g \cdot \nabla\left(\Jap{v}^{2q}a[f]\right) \dd v\\
&\label{eq: Q1shortttime beta_2=k} \le -\frac12 \int \nabla\left(\bdary^\beta g\right)^2 \cdot \nabla\left(\Jap{v}^{2q}a[f]\right) \dd v = \frac12 \int (\bdary^\beta g)^2 \Delta(a[f]\Jap{v}^{2q}) \dd v.
\end{align}
We record 
$$\Delta(a[f]\Jap{v}^{2q}) = (2+\gamma)h[f]\Jap{v}^{2q} + 4q \Jap{v}^{2q-2} \nabla a[f] \cdot v+ a[f](4q(q-1)|v|^2\Jap{v}^{2q-4} + 6q\Jap{v}^{2q-2})$$
in dimension $d=3$. It suffices for our purposes to bound
$$|\Delta(a[f]\Jap{v}^{2q})| \lesssim_{q,\gamma}(h[f] + |\nabla a[f]| + a[f])\Jap{v}^{2q}.$$

From Lemma \ref{l: app1} we find $\|a[f]\|_{L^\infty} \lesssim \|f\|_{L^2_m}$ for $m> \frac72 + \gamma$. From Lemma \ref{l: derivative estimates a and h} we find $\|\nabla a[f]\|_{L^\infty} \lesssim \|f\|_{H^1_m}$ for $m> \frac72 +\gamma$. Finally, we use Lemma \ref{l: derivative estimates a and h} and the identity $h[f] = (2+\gamma)\Delta a[f]$ to estimate

$$\|h[f]\|_{L^\infty} \lesssim \begin{cases} \|f\|_{H^2} & \gamma = -3, \\ \|f\|_{H^2_m} & -3<\gamma<-2,\quad m>\frac72 +\gamma.\end{cases}$$

The Coulomb case is a simple application of the Sobolev embedding $H^2(\R^3) \subset L^\infty(\R^3)$. So, fixing $m > \frac72 + \gamma$, we arrive at the estimate
\begin{align}
    \int_{\R^3}  \Jap{v}^{2q} \bdary^\beta g a[f] \Delta \bdary^\beta g \dd v & \lesssim \|g\|_{H^{|\beta|}_q}^2\|f\|_{H^2_m} \le \|g\|_{H^k_q}^2 \|f\|_{H^k_m}.
\end{align}

\textbf{Case Three: Estimating the $\mathcal{Q}_1$ terms when $|\beta_2| = k-1$.}
In this case, it is convenient to remember that we are not simply estimating $\int_{\R^3} \bdary^\beta g \mathcal{Q}(\bdary^{\beta_1}f, \bdary^{\beta_2}g) \dd v$ for fixed $\beta_1$ and $\beta_2$, but rather we are estimating the sum over all $\beta_1$ and $\beta_2$ with $\bdary^\beta = \bdary^{\beta_1}\bdary^{\beta_2}$. In particular, for $|\beta_2|=k-1$, writing $\bdary^\beta = \bdary_i \bdary^{\beta_2}$, it suffices to estimate terms of the form
$$\sum_i \int_{\R^3}  \Jap{v}^{2q}\bdary_i \bdary^{\beta_2}g \bdary_i a[f] \Delta \bdary^{\beta_2}g \dd v .$$

As we will see, keeping the sum allows us to introduce a useful Laplacian term. Let us define $G \coloneqq \bdary^{\beta_2}g$. Then we estimate
\begin{align}
    \sum_i \int_{\R^3}  \Jap{v}^{2q}\bdary_i G\bdary_i a[f] \Delta G \dd v & = \sum_{i,j} \int_{\R^3}  \Jap{v}^{2q}\bdary_i G\bdary_i a[f] \bdary_{jj} G \dd v\\
    &=-\sum_{i,j} \int_{\R^3}  \bdary_j\left(\Jap{v}^{2q}\bdary_i a[f]\right) \bdary_iG \bdary_j G \dd v- \frac12\sum_{i,j} \int_{\R^3}  \Jap{v}^{2q}\bdary_ia[f] \bdary_i(\bdary_jG)^2\dd v\\
    &\coloneqq J_1+J_2.
\end{align}
We handle $J_1$ and $J_2$ separately. Using the identity $\Delta a[f] = (2+\gamma)h[f]$, we find that for $m>\frac72+\gamma$ it holds that

\begin{align}
    J_2 &= \frac{2+\gamma}{2}\int_{\R^3}  \Jap{v}^{2q}|\nabla G|^2 h[f] \dd v + q\sum_{i,j} \int_{\R^3}  \Jap{v}^{2q-2}v_i \partial_i a[f] (\partial_jG)^2 \dd v\\
    &\lesssim \|g\|_{H^k_q}^2\|f\|_{H^2_m}
\end{align}
using the bounds on $\|h[f]\|_{L^\infty}$ and $\|a[f]\|_{L^\infty}$ from the previous step. The first term $J_1$ we handle by writing $\bdary^\alpha a[f] = a[\bdary^\alpha f]$ and using the same lemmas as in the previous step. We find
\begin{align}
    J_1 &\lesssim_q \sum_{i,j}\int_{\R^3}  \Jap{v}^{2q}|\nabla G|^2 (|a[\bdary_{ij}f]| + |a[\bdary_i f]|)\dd v\\
    &\lesssim \|g\|_{H^k_q}^2\|f\|_{H^2_m}
\end{align}
for any $m>\frac72+\gamma$. This concludes the proof of \eqref{eq: Q1estimateshorttime} in all cases. 
\\\par 
\textbf{Case Four:  Estimating the $\mathcal{Q}_2$ terms.} We prove
\begin{equation}
    \label{eq: Q2estimateshorttime} 
    \int_{\R^3} \Jap{v}^{2q} \bdary^\beta g \mathcal{Q}_2(f,g)
    =-(2+\gamma) 
    \int_{\R^3}  \Jap{v}^{2q} \bdary^\beta g \bdary^{\beta_1}h[f] \bdary^{\beta_2}g \dd v \lesssim \|g\|_{H^k_q}^2 \|f\|_{H^k_m}.
\end{equation}

Since $k\ge 3$, either $|\beta_1| \le k-2$ or $|\beta_2| \le k-2$. In the former case, we use Lemma \ref{l: app1} for $m > 3+\gamma $, as well as the weighted Sobolev embedding $H^2_m(\R^3) \subset L^\infty_m(\R^3)$ from Lemma \ref{l: app5}, to find
\begin{align}
    \int_{\R^3}  \Jap{v}^{2q} \bdary^\beta g \bdary^{\beta_1}h[f] \bdary^{\beta_2}g \dd v &\lesssim \|g\|_{H^{|\beta|}_q} \|g\|_{H^{|\beta_2|}_q} \|\bdary^{\beta_1}f \ast|\cdot|^\gamma\|_{L^\infty} \dd v\\
    &\lesssim \|g\|_{H^k_q}^2 \|\bdary^{\beta_1}f\|_{L^\infty_m} \lesssim \|g\|_{H^k_q}^2 \| f\|_{H^k_m}.
\end{align}
In the last line we used $|\beta_1| + 2\le k$. Note that for $\gamma=-3$, we are getting the same estimate using $\|\bdary^{\beta_1}f\|_{L^\infty} \lesssim \|\bdary^{\beta_1}f\|_{H^2} \le \|f\|_{H^k}.$ This concludes the proof of \eqref{eq: Q2estimateshorttime} for $|\beta_1|\le k-2$.

Now, if $|\beta_2| \le k-2$, we use  Lemma \ref{l: app3} with $m> 3+\gamma$ to obtain
\begin{align}
    \int_{\R^3}  \Jap{v}^{2q} \bdary^\beta g \bdary^{\beta_1}h[f] \bdary^{\beta_2}g \dd v & \le \| g\|_{H^{|\beta|}_q} \|\bdary^{\beta_2}g\|_{L^\infty_q}\|\bdary^{\beta_1}h[f]\|_{L^2}\\
    &\lesssim \|g\|_{H^k_q} \|g\|_{H^{2+|\beta_2|}_q} \|\bdary^{\beta_1}f\|_{L^2_m} \lesssim \|g\|_{H^k_q}^2\|f\|_{H^k_m}.
\end{align}
Here we used the Sobolev embedding $H^2_q(\R^3) \subset L^\infty_q(\R^3)$ from Lemma \ref{l: app5}. The same inequality holds in the Coulomb case $\gamma=-3$, as one trivially has $\|\bdary^{\beta_1}h[f]\|_{L^2} \le \|f\|_{H^{|\beta_1|}} \le \|f\|_{H^k} \le \|f\|_{H^k_m}.$
\\\par In conclusion, putting together \eqref{eq: snelsonestimate1}, \eqref{eq: snelsonestimate2}, \eqref{eq: Q1estimateshorttime}, and \eqref{eq: Q2estimateshorttime}, we find that for any multi-index $|\beta|\le k$, 

$$\frac{d}{dt} \int_{\R^3}  \Jap{v}^{2q} (\partial^\beta g)^2 \dd v\lesssim_{q,\gamma} \sigma\|g\|_{H^k_q}^2\|f\|_{H^k_m}+(1-\sigma)\|g\|_{H^k_q}^2 + 
 \|R\|_{L^2_q}^2.$$
Summing over $|\beta|\le k$ and bounding the factors $\sigma$ and $1-\sigma$ by $1$, we obtain
$$\frac{d}{dt}\|g\|_{H^k_q}^2  \lesssim (1+\|f\|_{H^k_m})\|g\|_{H^k_q}^2 + \|R\|_{L^2_q}^2,$$

and the estimate \eqref{eq: shorttimeestimate} follows from Gr\"onwall's inequality.
\end{proof}
Note that we also have the following estimate for $\mathcal{Q}_{KS}$.
 \begin{lem}   \label{l: Q_{KS}}
 For any $q>0$ and $\gamma \in [-3,-2),$ there exists $m> \frac72+\gamma$ and a constant $C = C(m, \gamma)<\infty$ such that
    \begin{equation}
        \norm{\mathcal{Q}_{KS}(f,g)}_{L^2_q}  \le C\norm{f}_{H^2_m} \norm{g}_{H^2_q}.
    \end{equation}
 \end{lem}
 \begin{proof}
 By the triangle inequality,
     \begin{equation}
           \norm{\mathcal{Q}_{KS}(f,g)}_{L^2_q}  \le \norm{\mathcal{Q}_1(f,g)}_{L^2_q}+\norm{\mathcal{Q}_2(f,g)}_{L^2_q}.
     \end{equation}
     Then by Lemmas \ref{l: app1} and \ref{l: app5}
     \begin{equation}
         \norm{\mathcal{Q}_1(f,g)}_{L^2_q} \le \norm{a[f]}_{L^\infty} \norm{g}_{H^2_q} \lesssim \norm{f}_{L^2_m} \norm{g}_{H^2_q}
     \end{equation}
     and
     \begin{equation}
         \norm{\mathcal{Q}_2(f,g)}_{L^2_q} \le \norm{h[f]}_{L^\infty} \norm{g}_{L^2_q} \lesssim  \norm{f}_{L^\infty_m} \norm{g}_{L^2_q} \lesssim \norm{f}_{H^2_m} \norm{g}_{L^2_q}.
     \end{equation}
 \end{proof}
 Finally we may apply the method of continuity to solve $\partial_t g = \mathcal{Q}_{KS}(f,g)$ with the following lemma. 
\begin{lem}\label{l: cont}
 Let $T>0, m>\frac72+\gamma, q>0$ be arbitrary. For any $f \in L^\infty([0,T], H_m^3(\R^3))$ and $g_{\text{in}} \in H_q^3(\R^3)$ with $f, g_{\text{in}}$ nonnegative, there exists a solution 
 $$g \in L^\infty([0,T], H_q^{3}(\R^3)) \cap W^{1,\infty}([0,T], L_q^2(\R^3))$$ 
 to the initial value problem 
 $$\begin{cases} \label{eq: IVP} \partial_t g = \mathcal{Q}_{KS}(f,g), \\ g(0,v) = g_{\text{in}}(v).\end{cases} $$
Furthermore, $g\ge 0$. 
\end{lem}
\begin{proof}
    The statement and proof of this lemma are the exact same as that of Lemma $6.4$ in \cite{snelson2023global} , since we have the same estimates and work with the same spaces. Although, there seems to a typo in the \cite{snelson2023global}, so we clarify the spaces that we use, which are
    \begin{equation}
        X_T \coloneqq L^\infty([0,T], H_q^{3}(\R^3)) \cap W^{1,\infty}([0,T], L_q^2(\R^3))
    \end{equation}
    and
    \begin{equation}
        Y_T \coloneqq L^\infty([0,T];L^2_q(\R^3)) \times H^3_q(\R^3).
    \end{equation}
From here, the proof is verbatim the same.
\end{proof}
Finally, we complete the proof of Proposition \ref{prop: localexistenceHK}.
\begin{proof}[Proof of Proposition \ref{prop: localexistenceHK}.]
    The proof is essentially the same as the proof of Theorem $1.2$ in \cite{snelson2023global}. Let $C$ be the constant from Lemma \ref{lem: local}, and define 
    $$T = \frac{\log 2}{2C \|f_{\text{in}}\|_{H^3_q}}.$$
    Define $f_0 = f_{\text{in}}$ and solve the iteration 
    $$\begin{cases} \label{eq: iteration} \partial_t f_{n+1} = \mathcal{Q}_{KS}(f_n, f_{n+1}), \\ f(0) = f_{\text{in}}\end{cases}$$
    on $[0,T]$. The existence and nonnegativity of $f_{n+1} \ge 0$ follows from Lemma \ref{l: cont}. Via induction, one proves in the exact same way as in \cite{snelson2023global} that 
    $$\|f_n\|_{L^\infty([0,T]; H_q^3)} \le 2 \|f_{\text{in}}\|_{H_q^3}$$
    for all $n$.

    Indeed, the estimate follows by applying Lemma \ref{lem: local} with $R=0$, $\sigma=1$, and $m=q$. Next, one proves that $f_n$ is bounded in $W^{1,\infty}([0,T]; H_{q}^1(\R^3))$ in an identical fashion. Indeed, for $|\beta| \le 1$,
    \begin{equation}
        \norm{\bdary^\beta \partial_t f_{n+1}(t)}_{L^2_q(\R^3)} \le C \left( \norm{f_{n}}_{H^{|\beta|+2}_q(\R^3)} \norm{f_{n+1}}_{H^{|\beta|+2}_q(\R^3)}     \right) \le C  \norm{f_{n}}_{H^{3}_q(\R^3)} \norm{f_{n+1}}_{H^{3}_q(\R^3)}   .
    \end{equation}
    This allows us to extract a subsequence converging weak$-\ast$ in $L^\infty([0,T]; H_{q}^{3}(\R^3))$, strongly in $L^\infty([0,T]; H_{q}^2(\R^3))$, and pointwise almost everywhere to a limit $f \in L^\infty([0,T]; H_q^2(\R^3))$.
    To see that $f$ is a solution, let $\varphi \in C_c^\infty((0,T)\times \R^3)$ and integrate by parts the equation for $f_{n+1}$. We have 
    \begin{align}
        -\iint_{[0,T]\times \R^3 }f_{n+1}\partial_t \varphi \dd v \dd t
        &= \iint_{[0,T]\times \R^3 } ( \varphi \mathcal{Q}_{KS}(f_n, f_n) + \varphi \mathcal{Q}_{KS}(f_n, f_{n+1}-f_n) ) \dd v \dd t\\
        &= \frac 1 2 \iint_{[0,T]\times \R^3 \times \R^3 }  (\bdary_{w_i} \varphi-\bdary_{v_i}\varphi) |v-w|^{2+\gamma} (\partial_{v_i}-\partial_{w_i})(f_n(v)f_n(w))  \ dw \dd v \dd t\\
        &+ \iint_{[0,T]\times \R^3 } \varphi \mathcal{Q}_{KS}(f_n, f_{n+1}-f_n)  \dd v \dd t.
    \end{align}
The latter term is controlled by 
\begin{align}
    \iint_{[0,T]\times \R^3 } \varphi \mathcal{Q}_{KS}(f_n, f_{n+1}-f_n) ) \dd v \dd t &\le\|\varphi\|_{L^\infty([0,T]; L^2)} \|\mathcal{Q}_{KS}(f_n, f_{n+1}-f_n)\|_{L^\infty([0,T];L^2_q)}\\
    &\le C \|\varphi\|_{L^\infty([0,T]; L^2)} \|f_n\|_{L^\infty([0,T]; H^2_q)}\| f_{n+1}-f_n\|_{L^\infty([0,T];H^2_q)} \to 0
\end{align}
thanks to Lemma \ref{l: Q_{KS}}, applied with $q=m>\frac72+\gamma$. Passing to the limit in the first integral and then integrating by parts and using symmetry, we find
\begin{equation} \label{eq: limit eq}
    -\iint_{[0,T]\times \R^3} f \bdary_t \varphi \dd v \dd t = \iint_{[0,T]\times \R^3} \varphi \mathcal{Q}_{KS}(f,f) \dd v \dd t.
\end{equation}
\end{proof}
We have the uniform bounds for $f_n$ in $W^{1,\infty}([0,T];H^1_q(\R^3)$ and $L^\infty([0,T];H^3_q(\R^3))$. Interpolating these two spaces and applying Morrey's embedding gives that $f_n$ is uniformly $C^\alpha_t C^\beta_v$ H\"older continuous for some $0<\alpha,\beta<1$. Hence, the limit $f$ is also continuous, which implies that $f$ attains its initial value $f_\text{in}$ continuously. Since \eqref{eq: limit eq} holds for every smooth $\varphi$, we can conclude that $f$ is differentiable in $t$. This proves that $f$ solves the nonlinear equation \eqref{eq: Kriegerpotential} pointwise. It also follows that $f\in W^{1,\infty}([0,T]; H^1_q(\R^3)).$\\
Uniqueness follows by an identical Gr\"onwall's inequality argument to that in \cite{snelson2023global}. Moreover, the final statement of the theorem about the propagation of weighted Sobolev norms for initial data with higher  regularity and decay is the exact same. For this reason we omit the proofs. 
\bibliographystyle{abbrv}
\bibliography{cite}

\begin{thebibliography}{10}

\bibitem{carrapatoso2017landau}
K.~Carrapatoso and S.~Mischler.
\newblock Landau equation for very soft and coulomb potentials near maxwellians, 2017.

\bibitem{desvillettes2005equilibrium}
L.~Desvillettes and C.~Villani.
\newblock On the trend to global equilibrium for spatially inhomogeneous kinetic systems: the {B}oltzmann equation.
\newblock {\em Invent. Math.}, 159(2):245--316, 2005.

\bibitem{968a84f7-1561-3553-9573-cf52054db1c4}
A.~Friedman and B.~McLeod.
\newblock Blow-up of positive solutions of semilinear heat equations.
\newblock {\em Indiana University Mathematics Journal}, 34(2):425--447, 1985.

\bibitem{golding2024global}
W.~Golding, M.~Gualdani, and A.~Loher.
\newblock Global smooth solutions to the landau-coulomb equation in $l^{3/2}$, 2024.

\bibitem{golding2024localintime}
W.~Golding and A.~Loher.
\newblock Local-in-time strong solutions of the homogeneous landau-coulomb equation with $l^p$ initial datum, 2024.

\bibitem{gressman2012ks}
P.~T. Gressman, J.~Krieger, and R.~M. Strain.
\newblock A non-local inequality and global existence.
\newblock {\em Adv. Math.}, 230(2):642--648, 2012.

\bibitem{gualdani2022hardy}
M.~Gualdani and N.~Guillen.
\newblock Hardy's inequality and the isotropic {L}andau equation.
\newblock {\em J. Funct. Anal.}, 283(6):Paper No. 109559, 25, 2022.

\bibitem{gualdani2018evensol}
M.~P. Gualdani and N.~Zamponi.
\newblock Global existence of weak even solutions for an isotropic {L}andau equation with {C}oulomb potential.
\newblock {\em SIAM J. Math. Anal.}, 50(4):3676--3714, 2018.

\bibitem{guillen2023global}
N.~Guillen and L.~Silvestre.
\newblock The landau equation does not blow up.
\newblock {\em arxiv preprint arXiv:2311.09420}, 2023.

\bibitem{guo2002periodic}
Y.~Guo.
\newblock The {L}andau equation in a periodic box.
\newblock {\em Comm. Math. Phys.}, 231(3):391--434, 2002.

\bibitem{henderson2019local}
C.~Henderson, S.~Snelson, and A.~Tarfulea.
\newblock Local existence, lower mass bounds, and a new continuation criterion for the {L}andau equation.
\newblock {\em J. Differential Equations}, 266(2-3):1536--1577, 2019.

\bibitem{henderson2020local}
C.~Henderson, S.~Snelson, and A.~Tarfulea.
\newblock Local solutions of the {L}andau equation with rough, slowly decaying initial data.
\newblock {\em Ann. Inst. H. Poincar\'{e} C Anal. Non Lin\'{e}aire}, 37(6):1345--1377, 2020.

\bibitem{krieger2012ks}
J.~Krieger and R.~M. Strain.
\newblock Global solutions to a non-local diffusion equation with quadratic non-linearity.
\newblock {\em Comm. Partial Differential Equations}, 37(4):647--689, 2012.

\bibitem{mckean1966kac}
H.~P. McKean, Jr.
\newblock Speed of approach to equilibrium for {K}ac's caricature of a {M}axwellian gas.
\newblock {\em Arch. Rational Mech. Anal.}, 21:343--367, 1966.

\bibitem{silvestre2017upper}
L.~Silvestre.
\newblock Upper bounds for parabolic equations and the landau equation.
\newblock {\em Journal of Differential Equations}, 262(3):3034--3055, 2017.

\bibitem{snelson2023global}
S.~Snelson.
\newblock Global existence for an isotropic modification of the boltzmann equation.
\newblock 2023.

\bibitem{strain2006almostexponential}
R.~M. Strain and Y.~Guo.
\newblock Almost exponential decay near {M}axwellian.
\newblock {\em Comm. Partial Differential Equations}, 31(1-3):417--429, 2006.

\bibitem{strain2008exponential}
R.~M. Strain and Y.~Guo.
\newblock Exponential decay for soft potentials near {M}axwellian.
\newblock {\em Arch. Ration. Mech. Anal.}, 187(2):287--339, 2008.

\bibitem{toscani1992boltzmann}
G.~Toscani.
\newblock New a priori estimates for the spatially homogeneous {B}oltzmann equation.
\newblock {\em Contin. Mech. Thermodyn.}, 4(2):81--93, 1992.

\bibitem{villani1998fisherboltzmann}
C.~Villani.
\newblock Fisher information estimates for {B}oltzmann's collision operator.
\newblock {\em J. Math. Pures Appl. (9)}, 77(8):821--837, 1998.

\bibitem{villani2000fisherlandau}
C.~Villani.
\newblock Decrease of the {F}isher information for solutions of the spatially homogeneous {L}andau equation with {M}axwellian molecules.
\newblock {\em Math. Models Methods Appl. Sci.}, 10(2):153--161, 2000.

\bibitem{villani2000fisherheat}
C.~Villani.
\newblock A short proof of the ``concavity of entropy power''.
\newblock {\em IEEE Trans. Inform. Theory}, 46(4):1695--1696, 2000.

\end{thebibliography}

\end{document}